\documentclass{siamart1116}
\usepackage{amsfonts}
\usepackage{amsmath,amssymb}
\usepackage[all]{xy}
\usepackage{tikz}
\usepackage{enumitem,url}
\usetikzlibrary{trees}
\usetikzlibrary{arrows,shapes,snakes,automata,backgrounds,petri}
\usepackage{accents}
\usepackage{bm}
\newlength{\dhatheight}
\newcommand{\doublehat}[1]{%
    \settoheight{\dhatheight}{\ensuremath{\hat{#1}}}%
    \addtolength{\dhatheight}{-0.35ex}%
    \hat{\vphantom{\rule{1pt}{\dhatheight}}%
    \smash{\hat{#1}}}}

\mathchardef\mhyphen="2D 

\newtheorem{thm}{Theorem}
\newtheorem{thm1}{Theorem}[section]
\newtheorem{lemma1}[thm1]{Lemma}

\newtheorem{cor}[thm1]{Corollary}
\newtheorem{def1}[thm1]{Definition}
\newtheorem{example}[thm1]{Example}
\newtheorem{notation}[thm1]{Notation}

\theorembodyfont{\normalfont} 
\newtheorem{remark}[thm1]{Remark}

\newenvironment{myproof}[1] {{\em Proof of {#1}. }}{\hfill$\square$}

\definecolor{light-gray}{gray}{0.9}
\definecolor{mygray}{gray}{0.8}

\title{The inheritance of nondegenerate multistationarity in chemical reaction networks}
\author{Murad Banaji\footnotemark[1] \and Casian Pantea\footnotemark[2]}
\begin{document}
\maketitle

\renewcommand{\thefootnote}{\fnsymbol{footnote}}

\footnotetext[1]{Middlesex University, London, Department of Design Engineering and Mathematics. {\tt m.banaji@mdx.ac.uk}.}
\footnotetext[2]{West Virginia University, Department of Mathematics. {\tt cpantea@math.wvu.edu}. Work supported by NSF DMS Grant 1517577 {\em Multistationarity and Oscillations in Biochemical Reaction Networks.}}
\renewcommand{\thefootnote}{\arabic{footnote}}

\begin{abstract}
We study how the properties of allowing multiple positive nondegenerate equilibria (MPNE) and multiple positive linearly stable equilibria (MPSE) are inherited in chemical reaction networks (CRNs). Specifically, when is it that we can deduce that a CRN admits MPNE or MPSE based on analysis of its subnetworks? Using basic techniques from analysis we are able to identify a number of situations where MPNE and MPSE are inherited as we build up a network. Some of these modifications are known while others are new, but all results are proved using the same basic framework, which we believe will yield further results. The results are presented primarily for mass action kinetics, although with natural, and in some cases immediate, generalisation to other classes of kinetics. 
\end{abstract}

\section{Introduction and context of the paper}

The idea of studying ``network motifs'' in biological systems, namely simple, frequently occurring, structures from which some properties of the system may be inferred, has attracted considerable recent interest (for example \cite{alon2007,Grochow2007}). The reason is natural: systems biology proceeds dually by elucidating the functioning of subsystems while at the same time attempting to fit together these pieces to explain large-scale behaviours in biological systems. Crucial to this study is a need to know how behaviours in a dynamical system are affected by embedding this system in some larger system in well-defined ways. Some rigorous results in this area come from the literature on hyperbolicity and structural stability (\cite{Ano95}, for example), and on monotone control systems (\cite{angeli}, for example). A closely related type of question is when a behaviour known to occur in some member of a {\em family} of dynamical systems $\mathcal{F}$ can be guaranteed to occur in some member of a new family $\mathcal{F}'$ related to $\mathcal{F}$ in some natural sense. 

Our goal is understanding the inheritance of {\em multiple positive nondegenerate equilibria} (MPNE) and {\em multiple positive linearly stable equilibria} (MPSE) in chemical reaction networks (CRNs) which form the main component of most biological models. The literature on multiple equilibria in CRNs is extensive and we do not attempt a survey here; the recent papers \cite{muellerAll}, \cite{banajipantea}, and \cite{Joshi.2015aa} provide some references to key strands of this work. But most work, including our own, has focussed on conditions for {\em precluding} multistationarity; in some ways the opposite of the question treated here.

An exception is the paper by Joshi and Shiu \cite{joshishiu}, which inspired this paper. The authors treated questions of inheritance of multistationarity primarily, but not exclusively, focussed on fully open CRNs. They highlighted that in studying multistationarity (and potentially other nontrivial behaviours in CRNs), we need to identify a relevant partial order on the set of all CRNs, and then search for minimal elements w.r.t. this partial order, which they termed {\em atoms}. This partial order $\leq$ should be such that if some CRN $\mathcal{R}$ displays the behaviour of interest, and $\mathcal{R} \leq \mathcal{R}'$, then $\mathcal{R}'$ must display this behaviour. In the context of multistationarity in fully open CRNs, and with certain kinetic assumptions, there is a nice result: MPNE and MPSE are inherited under the induced subnetwork partial order. This result (Corollary~\ref{coropeninduced} below) was largely known to Joshi and Shiu (Corollary~4.6 of \cite{joshishiu}). In fact, the results here allow considerably stronger claims about the inheritance of MPNE and MPSE in fully open networks (see Remark~\ref{remFObest}).

Results about fully open CRNs do not generalise easily to arbitrary CRNs: building larger CRNs from smaller ones in natural ways such as adding reactions can both introduce and destroy multistationarity. Nevertheless, it is possible to formulate several operations which together define a partial order $\preceq$ on CRNs such that if $\mathcal{R}$ displays MPNE (resp., MPSE) and $\mathcal{R} \preceq \mathcal{R}'$ then $\mathcal{R}'$ displays MPNE (resp., MPSE), and it is this task which we begin in this paper. For example, we can add new reactions involving some new species while preserving MPNE and MPSE provided certain conditions are met (Theorem~\ref{thmblockadd}). Feliu and Wiuf \cite{feliuwiufInterface2013} show that multistationarity can survive when a reaction is ``split'' with the introduction of a new intermediate species. We have a result involving splitting reactions and inserting intermediate complexes (Theorem~\ref{thmintermediates}) which reproduces and generalises some elements of their results.

Availability of a partial order relevant to MPNE or MPSE couples naturally with the task of characterising minimal networks admitting these behaviours w.r.t. this partial order. Although minimal networks need not be small, in the sense of having few species, few reactions, or reactions of low molecularity, a first step is naturally to identify {\em small} minimal networks with the desired behaviour. We do not undertake this task here, but it was begun for the induced subnetwork partial order in \cite{Joshi.2013aa} and continued in the recent work \cite{JoshiShiu2016}.

Our proofs largely rely on the implicit function theorem. We do not need degree theory \cite{soule, Conradi.2016aa}, or homotopy theory \cite{CraciunHeltonWilliams}, or any nontrivial algebra, even though many of the resulting claims can be seen as claims about the zeros of polynomial equations. Our local approach, though powerful, necessarily has limitations discussed in the conclusions. 

The paper is laid out as follows. After introducing some background on CRNs in Section~\ref{secbackground} we present an extended example (Section~\ref{secextended}) beginning with a CRN which displays MPNE and illustrating some of the many ways one might ``build up'' a reaction network, some of which will be proved to preserve MPNE and some of which destroy it. This foreshadows the theorems to follow, while also highlighting their limitations. We then present in Section~\ref{secthms} the results, some known and some new, but defer the proofs until Section~\ref{secproofs}, after illustrating application of the results on a biologically important CRN in Section~\ref{secbioexample}. Finally, we present some discussion and conclusions. A consequence of the analysis in Section~\ref{secbioexample} is previewed in Figure~\ref{fig:MAPK}.

\begin{figure}[h]
\begin{center}
\resizebox{8cm}{!}{
\begin{tikzpicture}
[scale=2.4, 
place/.style={circle,draw=blue!50,fill=blue!20,thick,inner sep=0pt,minimum size=5.5mm},
enzyme/.style={circle,draw= blue!50,fill=white,thick,inner sep=0pt,minimum size=5.5mm},
pre/.style={<-,shorten <=1pt,>=stealth',semithick},
post/.style={->,shorten >=1pt,>=stealth',semithick}]


\node (MAPK) at (-1, 0) {\bf{MAPK}};
\node (MAPK-pp) at (1,0) {\bf{MAPK-pp}};
\node (MAPK-p) at (0,0) {\bf{MAPK-p}}
edge [bend left = 40, post,ultra thick]  (MAPK)
edge [bend right = 40, pre,ultra thick] (MAPK-pp)
edge [bend right = 40, pre,ultra thick]  (MAPK)
edge [bend left = 40, post,ultra thick] (MAPK-pp);

\node (MKK) at (-2, 1) {MKK};
\node (MKK-pp) at (0,1) {\bf{MKK-pp}};
\node (MKK-p) at (-1,1) {MKK-p}
edge [bend left = 40, post]  (MKK)
edge [bend right = 40, pre] (MKK-pp)
edge [bend right = 40, pre]  (MKK)
edge [bend left = 40, post] (MKK-pp);

\node (MKKK) at (-2, 2) {MKKK};
\node (MKKK-p) at (-1,2) {MKKK-p}
edge [bend left = 40, post]  (MKKK)
edge [bend right = 40, pre]  (MKKK);

\node (Ras/MKKKK) at (-1.5, 2.8) {E$_1$};

\node (F1) at (-1.46, 1.67) {$\text{F}_1$};
\node (F2) at (-1.46, .67) {$\text{F}_2$};
\node (F22) at (-.46, .67) {$\text{F}_2$};
\node (F3) at (-.46, -.35) {$\text{F}_{3}$};
\node (F32) at (.54, -.35) {$\text{F}_{3}$};

\draw [semithick, ->,>=stealth', rounded corners=1.5mm, ultra thick]  (0,.9) -- (0,0.55) -- (0.5,0.55) -- (0.5,0.27);
\draw [semithick, ->,>=stealth', rounded corners=1.5mm, ultra thick]  (0,.9) -- (0,0.55) -- (-0.5,0.55) -- (-0.5,0.27);

\draw [semithick, ->,>=stealth', rounded corners=1.5mm]  (-1,1.9) -- (-1,1.55) -- (-0.5,1.55) -- (-0.5,1.27);
\draw [semithick, ->,>=stealth', rounded corners=1.5mm]  (-1,1.9) -- (-1,1.55) -- (-1.5,1.55) -- (-1.5,1.27);

\draw [semithick, ->,>=stealth', rounded corners=1mm]  (-1.5,2.7) -- (-1.5,2.27);

\draw [densely dashed,semithick, ->,>=stealth', rounded corners=1.5mm] (1,0.1) -- (1,2.5) -- (-1.47,2.5) ;

\end{tikzpicture}
}
\caption{MPSE in the Huang-Ferrell MAPK cascade with negative feedback \cite{Huang.1996aa} can be inferred from MPSE in the subnetwork in bold. The full analysis is carried out in Section~\ref{secbioexample}.}\label{fig:MAPK}
\end{center}
\end{figure}
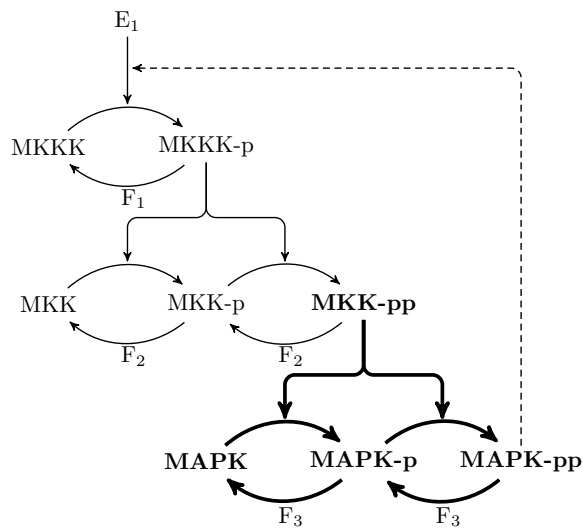

\section{Background on CRNs}
\label{secbackground}
We provide only a brief, informal introduction to CRNs, and the reader is referred to \cite{banajipantea, banajiCRNcount, banajiCRNosci} for further detail.

A {\em complex} is a formal linear combination of chemical species. Given a list of species $X = (X_1, \ldots, X_n)$, and a nonnegative integer vector $a \in \mathbb{R}^n$, we write $a \cdot X$ for the complex $a_1X_1 + a_2 X_2 + \cdots + a_nX_n$. The zero complex $0x_1 + \cdots + 0x_n$ will be denoted $0$. An irreversible {\em reaction} is an ordered pair of complexes, the {\em source} and {\em target} complexes. A CRN consists of a set of species, and a set of irreversible reactions involving these species. A CRN involving species $(X_1, \ldots, X_n)$ is {\em fully open} if it includes all the inflow-outflow reactions $0 \rightleftharpoons X_i$ ($i = 1, \ldots, n$). The {\em fully open extension} of a CRN $\mathcal{R}$ is created by adjoining to $\mathcal{R}$ any such reactions which are absent from $\mathcal{R}$.

\begin{remark}[Forbidden CRNs]
\label{remforbid}
Although it is not required for any results here, and so is not assumed, it is common to forbid from the definition of a CRN reactions with the same complex on left and right hand sides, reactions which figure more than once, and species which participate in no reactions.
\end{remark}

The {\em Petri net graph (PN graph)} of a CRN $\mathcal{R}$, denoted $PN(\mathcal{R})$, is an edge-weighted bipartite digraph \cite{angelipetrinet}, closely related to other bipartite graphs associated with CRNs, notably SR and DSR graphs \cite{craciun1,banajicraciun2}. Its two vertex sets $V_S$ (species vertices) and $V_R$ (reaction vertices) correspond to the species and the reactions of $\mathcal{R}$, and given $u \in V_S$ and $v \in V_R$, there exists an arc $uv$ (resp., $vu$) with weight $w$ if and only if the species corresponding to $u$ occurs with stoichiometry $w$ as a reactant (resp., product) in reaction corresponding to $v$. 

\begin{example}[CRN and its PN graph]
We illustrate below the PN graph of $X+2Y \rightarrow 3Y$, $Y \rightarrow X \rightleftharpoons 0$, with the convention that arc-weights of $1$ are omitted:
\begin{center}
\begin{tikzpicture}[scale=1.2]
\fill[color=black] (1,0.5) circle (1.5pt);
\fill[color=black] (1,-0.5) circle (1.5pt);
\fill[color=black] (3,0.5) circle (1.5pt);
\fill[color=black] (3,-0.5) circle (1.5pt);

\node at (2,0) {$X$};
\node at (4,0) {$Y$};

\draw [<-, thick] (1.85,0.15) .. controls (1.7,0.3) and (1.5,0.5) .. (1.1,0.5);
\draw [->, thick] (1.85,-0.15) .. controls (1.7,-0.3) and (1.5,-0.5) .. (1.1,-0.5);

\draw [->, thick] (2.15,0.15) .. controls (2.3,0.3) and (2.5,0.5) .. (2.9,0.5);
\draw [<-, thick] (3.85,0.15) .. controls (3.7,0.3) and (3.5,0.5) .. (3.1,0.5);
\draw [<-, thick] (2.15,-0.15) .. controls (2.3,-0.3) and (2.5,-0.5) .. (2.9,-0.5);
\draw [->, thick] (3.85,-0.15) .. controls (3.7,-0.3) and (3.5,-0.5) .. (3.1,-0.5);

\draw[->, thick] (3.85,0.05) .. controls (3.5, 0.05) and (3.3, 0.2) .. (3.07,0.43);

\node at (3.35,0.05) {$\scriptstyle{2}$};
\node at (3.6,0.5) {$\scriptstyle{3}$};
\end{tikzpicture}
\end{center}

\end{example}

CRNs $\mathcal{R}_1$ and $\mathcal{R}_2$ are {\em isomorphic} if $PN(\mathcal{R}_1)$ and $PN(\mathcal{R}_2)$ are isomorphic as edge-labelled digraphs, via an isomorphism which preserves the bipartition into species and reaction vertices. In other words, the species and reactions of $\mathcal{R}_1$ can be (independently) relabelled to get $\mathcal{R}_2$. An equivalence class of labelled CRNs under isomorphism is termed an {\em unlabelled CRN}. An {\em induced subnetwork} (or ``embedded network'' in the terminology of \cite{joshishiu}) of a (labelled) CRN $\mathcal{R}$ is a CRN whose PN graph is a vertex-induced subgraph of $PN(\mathcal{R})$, namely can be obtained by deleting some vertices and all their incident arcs from $PN(\mathcal{R})$. The definition extends naturally to unlabelled CRNs and the relationship ``is an induced subnetwork'' defines a partial order on the set of unlabelled CRNs, whether or not the definition of a CRN includes additional restrictions as in Remark~\ref{remforbid}. Isomorphism is discussed in more detail in \cite{banajiCRNcount}. 

In order to discuss multiple equilibria, we need a brief discussion of ODE models of CRNs.

\begin{notation}[Nonnegative and positive vectors]
$\mathbb{R}^n_{\geq 0}$ refers to the nonnegative orthant in $\mathbb{R}^n$, namely $\{x \in \mathbb{R}^n\,|\, x_i \geq 0,\,\,i=1, \ldots, n\}$, while $\mathbb{R}^n_{\gg 0}$ refers to the positive orthant, namely $\{x \in \mathbb{R}^n\,|\, x_i > 0,\,\,i=1,\ldots, n\}$. For $x \in \mathbb{R}^n$, $x \geq 0$ (resp., $x > 0$, rep., $x \gg 0$) will mean that $x \in \mathbb{R}^n_{\geq 0}$ (resp., $x \in \mathbb{R}^n_{\geq 0}\backslash\{0\}$, resp. $x \in \mathbb{R}^n_{\gg 0}$). A vector is nonnegative (resp., positive) if each component is nonnegative (resp., positive).
\end{notation}

Consider $n$ chemical species $X = (X_1, \ldots, X_n)$ with concentrations $(x_1, \ldots, x_n) \in \mathbb{R}^n_{\geq 0}$, and a CRN $\mathcal{R}$ involving $r_0$ reactions between these species. Choose some order for the reactions and let $\Gamma_{ij}$ be the net change in $X_i$ when reaction $j$ occurs. The $n \times r_0$ integer matrix $\Gamma = (\Gamma_{ij})$ is termed the {\em stoichiometric matrix} of $\mathcal{R}$ and the $j$th column of $\Gamma$ is the {\em reaction vector} of reaction $j$. $\mathrm{im}\,\Gamma$ is the {\em stoichiometric subspace} of $\mathcal{R}$ and the nonempty intersection of a coset of $\mathrm{im}\,\Gamma$ with the nonnegative (resp., positive) orthant is a {\em stoichiometry class} (resp., {\em positive stoichiometry class}) of $\mathcal{R}$. We say that $p,q \in \mathbb{R}^n$ are {\em compatible} if $p-q \in \mathrm{im}\,\Gamma$. In spatially homogeneous, deterministic, continuous time models, the evolution of the species concentrations $x$ is governed the ODE:
\begin{equation}
\label{genCRN}
\dot x = \Gamma v(x)\,,
\end{equation}
where $v$ is the {\em vector of reaction rates} or {\em rate vector} of the CRN. We always assume that $v$ is defined and $C^1$ on $\mathbb{R}^n_{\gg 0}$ and belongs to the class of {\em positive general kinetics} \cite{banajipantea}: in brief, on the positive orthant, the rate of an irreversible reaction (i) is positive; (ii) depends only on the concentrations of its reactants; and (iii) increases strictly with the concentration of each reactant. Positive stoichiometry classes are locally invariant for (\ref{genCRN}) for positive general kinetics, and indeed under more general assumptions. Positive general kinetics includes mass action (MA) kinetics as a special case, and in fact all our inheritance results are presented to be consistent with the assumption that all reactions have MA kinetics. The emphasis on MA kinetics is for brevity and readability: we include remarks about other classes of kinetics for which the results hold -- often requiring only minor, formal modification of the proofs.

{\bf Equilibria: nondegeneracy and stability (MPNE and MPSE).} 

\begin{def1}[Reduced Jacobian matrices, the reduced determinant \cite{banajipantea}]
\label{notationreddet}
Consider a differentiable function $F\colon X\subseteq \mathbb{R}^n \to \mathbb{R}^n$ such that that $\mathrm{im}\,F \subseteq S \subseteq \mathbb{R}^n$, where $S$ is some $k$-dimensonal linear subspace of $\mathbb{R}^n$. Let $M$ be some matrix whose columns are a basis for $S$, so there is a unique differentiable function $\hat{F}\colon X \to \mathbb{R}^k$ s.t. $F(x) = M\hat{F}(x)$. We term $D_MF:= (D\hat{F})M$ the {\em reduced Jacobian matrix} of $F$ w.r.t. $M$. A different choice of basis for $S$, written as the columns of a matrix $M'$, leads to a reduced Jacobian matrix $D_{M'}F$ similar to $D_MF$. When interested only in the spectral properties of such reduced Jacobian matrices, we refer to the linear operator they represent as $D_SF$. We refer to $J_SF:=\mathrm{det}(D_SF)$ as the {\em reduced determinant} of $F$ w.r.t. $S$.
\end{def1}

A point $p \in \mathbb{R}^n_{\gg 0}$ is {\em nondegenerate} for (\ref{genCRN}) if the Jacobian determinant of the restricted vector field $\left.\Gamma v(\cdot)\right|_{p+\mathrm{im}\,\Gamma}$ is nonzero at $p$ (see also Definition~4 in \cite{craciun2}). Equivalently, $J_{\mathrm{im}\,\Gamma}(\Gamma Dv(p)) \neq 0$. $J_{\mathrm{im}\,\Gamma}(\Gamma Dv(p))$ can be computed in various ways \cite{banajipantea}. Most simply, $J_{\mathrm{im}\,\Gamma}(\Gamma Dv(p))$ is the sum of $r \times r$ principal minors of $\Gamma Dv(p)$ where $r=\mathrm{rank}\,\Gamma$. Alternatively, following Definition~\ref{notationreddet}, we may choose $\Gamma_0$ to be any matrix whose columns are a basis for $\mathrm{im}\,\Gamma$ and define $Q$ via $\Gamma = \Gamma_0Q$. Then $J_{\mathrm{im}\,\Gamma}(\Gamma Dv(p)) = \mathrm{det}(Q Dv(p) \Gamma_0)$.

A CRN evolving according to (\ref{genCRN}) with kinetics in some fixed class admits {\em multiple positive equilibria} (MPE) if there there exists {\em some} choice of rate function $v$ in this class and $p,q \gg 0$ satisfying $p-q \in \mathrm{im}\,\Gamma$ and $\Gamma v(p) = \Gamma v(q) = 0$. Note that it is implicit in the term ``multiple equilibria'' that the equilibria are compatible. The CRN admits {\em multiple, positive, nondegenerate equilibria} (MPNE) if, additionally, $p$ and $q$ can be chosen to be nondegenerate, namely $J_{\mathrm{im}\,\Gamma}(\Gamma Dv(p))\neq 0$ and $J_{\mathrm{im}\,\Gamma}(\Gamma Dv(q))\neq 0$. The CRN admits {\em multiple positive linearly stable equilibria} (MPSE) if, additionally, $p$ and $q$ can be chosen to be linearly stable w.r.t. their stoichiometry class, namely $D_{\mathrm{im}\,\Gamma}(\Gamma Dv(p))$ and $D_{\mathrm{im}\,\Gamma}(\Gamma Dv(q))$ are both Hurwitz stable (i.e., have eigenvalues in the open left half plane of $\mathbb{C})$. MPSE are clearly MPNE.

Given a CRN $\mathcal{R}$ admitting MPNE (resp., MPSE), we will present a number of ways in which we can modify $\mathcal{R}$ to create a new CRN $\mathcal{R}'$ in such a way that $\mathcal{R}'$ displays MPNE (resp., MPSE). This will generally involve adding more reactions and/or more species, or modifying reactions in some way. A similar project can be undertaken for other dynamical behaviours, in particular for CRNs which admit hyperbolic periodic orbits \cite{banajiCRNosci}, and we remark on this after certain proofs. 

\section{An extended example}
\label{secextended}
To motivate the results to follow and highlight their limitations we present an extended example beginning with a CRN $\mathcal{R}_0$ displaying MPNE with mass action kinetics and then performing various modifications to create CRNs $\mathcal{R}_1$ to $\mathcal{R}_{12}$. Each modification can either be proved to preserve MPNE using a result in the paper, or highlights some point about the scope of these theorems.

{\bf Base reaction network.} We begin with the following CRN (see also Remark 4.4 of \cite{JoshiShiu2016}). 
\begin{equation}
X+2Y \rightarrow 3Y, \quad Y \rightarrow X\,. \tag{\mbox{$\mathcal{R}_0$}}
\end{equation}
If $\mathcal{R}_0$ is given mass action kinetics with both rate constants set at $1$, we get the ODEs $\dot x = -xy^2+y,\,\,\dot y = xy^2-y$. Fixing any $C>2$ and defining $t_{\pm} = \frac{C \pm \sqrt{C^2-4}}{2}$, the reader may easily confirm that $p=(t_+,t_-)$ and $q=(t_-,t_+)$, are positive, nondegenerate, compatible equilibria of the system. That $\mathcal{R}_0$ admits MPNE is also an immediate application of Theorem 4.5 in \cite{JoshiShiu2016}.

{\bf Modification 1: adding new linearly dependent reactions preserves MPNE.} Consider adding to $\mathcal{R}_0$ some new reaction with reaction vector which is a linear combination of existing reaction vectors. For example we might enlarge $\mathcal{R}_0$ with the reaction $X \rightarrow Y$ to get
\begin{equation}
X+2Y \rightarrow 3Y, \quad Y \rightleftharpoons X\,. \tag{\mbox{$\mathcal{R}_1$}}
\end{equation}
As the stoichiometry classes do not change, intuition suggests that a small rate for the new reaction causes only a small perturbation to the vector field on any stoichiometry class, and hence $\mathcal{R}_1$ should admit MPNE. This is formalised as a general result (Theorem~\ref{thmnewdepreac}). 

{\bf Modification 2: adding new independent reactions may destroy MPNE.} Consider now adding to $\mathcal{R}_0$ the reversible reaction $0 \rightleftharpoons X$ with mass action kinetics, to get 
\begin{equation}
X+2Y \overset{\scriptstyle{k_1}}\longrightarrow 3Y, \quad Y \overset{\scriptstyle{k_2}}\longrightarrow X, \quad 0 \overset{\scriptstyle{k_3}}{\underset{\scriptstyle{k_4}}\rightleftharpoons} X\,. \tag{\mbox{$\mathcal{R}_2$}}
\end{equation}
The unique stoichiometry class of this system is now $\mathbb{R}^2_{\geq 0}$, and the unique positive equilibrium of this system is at $x=k_3/k_4,\,y=k_2k_4/(k_1k_3)$. 
Thus $\mathcal{R}_2$ forbids MPNE, and we conclude that adding reactions to a CRN may destroy the capacity for MPNE.

{\bf Modification 3: adding new independent reactions may allow multiple degenerate equilibria.} Consider adding to $\mathcal{R}_0$ the reaction $0 \rightleftharpoons X+Y$ with mass action kinetics, to get
\begin{equation}
X+2Y \overset{\scriptstyle{k_1}}\longrightarrow 3Y, \quad Y \overset{\scriptstyle{k_2}}\longrightarrow X, \quad 0 \overset{\scriptstyle{k_3}}{\underset{\scriptstyle{k_4}}\rightleftharpoons} X + Y\,. \tag{\mbox{$\mathcal{R}_3$}}
\end{equation}
The unique stoichiometry class of $\mathcal{R}_3$ is $\mathbb{R}^2_{\geq 0}$. 
If $k_2/k_1 \neq k_3/k_4$, this system has no equilibria at all. On the other hand if $k_2/k_1 = k_3/k_4$, then the set of equilibria consists of solutions to $xy = k_2/k_1$, and all these equilibria are degenerate. Thus the capacity for MPNE is destroyed, even though $\mathcal{R}_3$ still allows MPE for special choices of rate constants. 

{\bf Modification 4: adding new independent reactions may preserve MPNE.} Consider adding to $\mathcal{R}_0$ three irreversible reactions $X+2Y \rightarrow 2X+2Y$, $3X\rightarrow 4X$ and $X \rightarrow 0$, and choosing mass action rate constants as indicated:
\begin{equation}
2X+2Y \overset{\scriptstyle{1}}\longleftarrow X+2Y \overset{\scriptstyle{1}}\longrightarrow 3Y, \quad Y \overset{\scriptstyle{1}}\longrightarrow X \overset{\scriptstyle{4}}\longrightarrow 0, \quad 3X \overset{\scriptstyle{1}}\longrightarrow 4X\,. \tag{\mbox{$\mathcal{R}_4$}}
\end{equation}
The unique stoichiometry class of $\mathcal{R}_4$ is $\mathbb{R}^2_{\geq 0}$. 
$(p,1/p)$ and $(q, 1/q)$ for $p=\sqrt{2+\sqrt{3}}$ and $q=\sqrt{2-\sqrt{3}}$ are nondegenerate positive equilibria of the system. Thus $\mathcal{R}_4$ admits MPNE. However, there is no obvious {\em local} argument for this: the MPNE constructed for the new system are not in any sense derived as perturbations of original MPNE. 

{\bf Modification 5: adding inflows and outflows of all species preserves MPNE.} Suppose we add to $\mathcal{R}_0$ inflows and outflows of all species to get the fully open extension of $\mathcal{R}_0$:
\begin{equation}
X+2Y \rightarrow 3Y, \quad Y\rightarrow X, \quad X \rightleftharpoons 0 \rightleftharpoons Y\,. \tag{\mbox{$\mathcal{R}_5$}}
\end{equation}
The stoichiometric subspace becomes all of $\mathbb{R}^2$. This modification is well known to preserve MPNE. A variety of proofs are known and a very simple IFT-based argument is possible (Theorem~\ref{thmopenextension}).

{\bf Modification 6: adding a trivial species preserves MPNE.} Suppose we add a new species $Z$ into the reactions of $\mathcal{R}_0$ in a rather trivial way: the species always figures with the same stoichiometry on both sides of any reaction in which it participates. We might get for example:
\begin{equation}
X+2Y+Z \rightarrow 3Y+Z, \quad Y+2Z \rightarrow X+2Z\,.\tag{\mbox{$\mathcal{R}_6$}}
\end{equation}
The stoichiometry classes of $\mathcal{R}_6$ are still one dimensional and are parallel to the $x\mhyphen y$ plane. An easy general result (Theorem~\ref{thmtrivial}) confirms that $\mathcal{R}_6$ must admit MPNE.

{\bf Modification 7: adding new species may destroy MPNE.} Suppose we add a new species $Z$ into the reactions of $\mathcal{R}_0$ as follows:
\begin{equation}
X+2Y \rightarrow 3Y, \quad Y+Z \rightarrow X\,.\tag{\mbox{$\mathcal{R}_7$}}
\end{equation}
The concentration of $Z$ must be zero at any equilibrium, and thus $\mathcal{R}_7$ trivially forbids MPNE. It is also straightforward to find examples where adding a new species into some reactions of a CRN admitting MPNE results in a CRN with a unique positive equilibrium on each stoichiometric class. For example, assuming mass action kinetics, suppose we start with $\mathcal{R}_4$ above and add the new species $Z$ into two reactions to get
\begin{equation}
X+2Y \overset{\scriptstyle{k_1}}\longrightarrow 3Y, \quad X+2Y+2Z \overset{\scriptstyle{k_2}}\longrightarrow 2X+2Y, \quad Y \overset{\scriptstyle{k_3}}\longrightarrow X \overset{\scriptstyle{k_4}}\longrightarrow Z, \quad 3X \overset{\scriptstyle{k_5}}\longrightarrow 4X\,. \tag{\mbox{$\mathcal{R}'_4$}}
\end{equation}
The unique stoichiometry class of $\mathcal{R}'_4$ is all of $\mathbb{R}^3_{\geq0}$, and $\mathcal{R}'_4$ has a unique positive equilibrium at $(\alpha, \beta, \gamma)$ where $\alpha = \sqrt{k_4/2k_5}$, $\beta = k_3/k_1\alpha$, $\gamma = \sqrt{k_4/2k_2\beta^2}$.

{\bf Modification 8: adding new species may preserve MPNE.} Suppose, we add the new species $Z$ into the reactions of $\mathcal{R}_0$ as follows:
\begin{equation}
X+2Y+2Z \rightarrow 3Y+3Z, \quad Y+Z \rightarrow X\,.\tag{\mbox{$\mathcal{R}_8$}}
\end{equation}
Stoichiometry classes of $\mathcal{R}_8$ remain one dimensional and we can confirm that $(x,y,z)=(1,6,4)$, and $(x,y,z)=(3,4,2)$ are positive, nondegenerate, compatible equilibria of $\mathcal{R}_8$. Thus $\mathcal{R}_8$ admits MPNE. 

{\bf Modification 9: adding new reactions involving new species often preserves MPNE.} Consider adding to $\mathcal{R}_0$ a single reversible reaction $Y \rightleftharpoons Z$ involving a new species $Z$ to get:
\begin{equation}
X+2Y \rightarrow 3Y, \quad X \leftarrow Y \rightleftharpoons 2Z\,.\tag{\mbox{$\mathcal{R}_9$}}
\end{equation}
The evolution now occurs in $\mathbb{R}^3$, and the stoichiometric subspace of $\mathcal{R}_9$ is now $2$-dimensional. That $\mathcal{R}_9$ admits MPNE is a simple application of a much more general result (Theorem~\ref{thmblockadd}) allowing the addition of several new reactions involving several new species. We remark that Theorem~\ref{thmblockadd} includes a nondegeneracy condition which in this case is equivalent to the requirement that the new species must not occur with the same stoichiometry on both sides of the new reaction. The reader may confirm, for example, that 
\begin{equation}
X+2Y \rightarrow 3Y, \quad Y \rightarrow X, \quad Z \rightleftharpoons X+Z\,.\tag{\mbox{$\mathcal{R}'_9$}}
\end{equation}
forbids MPNE.


{\bf Modification 10: adding a new species into reactions while also adding its inflow and outflow preserves MPNE.} Suppose we add into some reactions of $\mathcal{R}_0$  a new species $Z$ (in any way) while also adding its inflow and outflow $0 \rightleftharpoons Z$. We may get, for example,
\begin{equation}
X+2Y \rightarrow 3Y, \quad Y+Z \rightarrow X, \quad 0 \rightleftharpoons Z\,.\tag{\mbox{$\mathcal{R}_{10}$}}
\end{equation}
Intuition suggests that sufficiently large inflow-outflow rates for $Z$ should effectively hold $Z$ constant and thus the net effect is effectively only to modify the rates of the reactions in which $Z$ participates. This argument is made precise in Theorem~\ref{thmnewwithopen}. 

{\bf Modification 11: splitting a reaction and inserting a complex involving a new species often preserves MPNE.} Suppose that we modify an irreversible reaction of $\mathcal{R}_0$ by adding a complex containing some new species $Z$, as an intermediate. For example we might get:
\begin{equation}
X+2Y \rightarrow X+Z \rightarrow 3Y\,, \quad Y \rightarrow X\,.\tag{\mbox{$\mathcal{R}_{11}$}}
\end{equation}
The evolution now occurs in $\mathbb{R}^3$, and the stoichiometric subspace is now $2$-dimensional. $\mathcal{R}_{11}$ admits MPNE as a consequence of a general result (Theorem~\ref{thmintermediates}) which allows adding several new species in several intermediate complexes.

{\bf Modification 12: splitting a reaction and inserting a complex involving no new species may destroy MPNE.} Suppose that we modify an existing irreversible reaction of $\mathcal{R}_0$ by adding a new complex comprised of existing species as an intermediate. For example we might get:
\begin{equation}
X+2Y \overset{\scriptstyle{k_1}}\longrightarrow 2X \overset{\scriptstyle{k_2}}\longrightarrow 3Y\,, \quad Y \overset{\scriptstyle{k_3}}\longrightarrow X\,.\tag{\mbox{$\mathcal{R}_{12}$}}
\end{equation}
The evolution occurs in $\mathbb{R}^2$, and the stoichiometric subspace is now $2$-dimensional. The reader may confirm that the only positive equilibrium of $\mathcal{R}_{12}$ is at
$x = \sqrt[3]{\frac{k_3^2}{k_1k_2}}, y=\sqrt[3]{\frac{k_3k_2}{k_1^2}}$. Thus $\mathcal{R}_{12}$ forbids MPNE.

{\bf Remarks on the modifications}

\begin{enumerate}[align=left,leftmargin=*]
\item If we add new ``dependent'' reactions, as in $\mathcal{R}_{1}$, then we can quite generally expect MPNE and MPSE to survive (Theorem~\ref{thmnewdepreac}). This is also the subject of Theorem~3.1 in \cite{joshishiu}.

\item If we add new ``independent'' reactions on existing species, as in Modifications 2--5, then the consequences are unclear as illustrated by $\mathcal{R}_2$, $\mathcal{R}_3$ and $\mathcal{R}_4$. However, MPNE and MPSE are preserved in the important and well known special case of taking the fully open extension as in $\mathcal{R}_{5}$ (Theorem~\ref{thmopenextension}). 
\item If we add species into reactions, without adding new reactions, then MPNE and MPSE are preserved if the added species figure trivially in reactions (Theorem~\ref{thmtrivial}). More generally, the consequences are unclear as illustrated by $\mathcal{R}_{7}$ and $\mathcal{R}_{8}$.
\item If we add new species to a CRN, then MPNE and MPSE survive in the important special case where we add an inflow and outflow reaction for each new species as in $\mathcal{R}_{10}$ (Theorem~\ref{thmnewwithopen}, and see also Theorem~4.2 in \cite{joshishiu}). If we add new reversible reactions involving some new species, without modifying any existing reactions (as in $\mathcal{R}_{9}$) then, with mild additional hypotheses, MPNE and MPSE survive (Theorem~\ref{thmblockadd}). 

\item Finally, $\mathcal{R}_{11}$ and $\mathcal{R}_{12}$ provide examples where we modify a network by adding intermediate complexes into reactions, effectively both deleting and adding reactions. If these involve new species, then under mild conditions, this preserves MPNE and MPSE, as in $\mathcal{R}_{11}$ (Theorem~\ref{thmintermediates}). The addition of intermediate complexes is also the subject of \cite{feliuwiufInterface2013}, and we comment later on the relationship with results in this paper. 
\end{enumerate}

\section{Results on the inheritance of MPNE and MPSE}
\label{secthms}

We gather together the main results, presenting the proofs later. Theorems~\ref{thmnewdepreac}-\ref{thmnewwithopen} are relatively easy and some elements of these results are well known. Theorems~\ref{thmblockadd}~and~\ref{thmintermediates} are harder. In each case, $\mathcal{R}$ is an unknown CRN on species $X = (X_1, \ldots, X_n)$ with positive general kinetics or mass action kinetics. The reactions of $\mathcal{R}$ are assumed, w.l.o.g., to be irreversible. $\mathcal{R}'$ is created from $\mathcal{R}$ by adding reactions and/or species (Theorems~\ref{thmnewdepreac}~to~\ref{thmblockadd}), or by splitting reactions (Theorem~\ref{thmintermediates}), and is given kinetics compatible with that of $\mathcal{R}$: if $\mathcal{R}$ has MA kinetics, then so does $\mathcal{R}'$, while if $\mathcal{R}$ has positive general kinetics, then so does $\mathcal{R}'$. Much more general assumptions on the kinetics are possible but are omitted for brevity: the proofs and surrounding remarks make clear how the kinetic assumptions can be loosened.

Theorem~\ref{thmnewdepreac} is closely related to Theorem~3.1 in \cite{joshishiu}.
\begin{thm}[Adding a dependent reaction]
\label{thmnewdepreac}
Suppose we create $\mathcal{R}'$ from $\mathcal{R}$, by adding to $\mathcal{R}$ a new irreversible reaction with reaction vector $\alpha$ which is a linear combination of reaction vectors of $\mathcal{R}$. If $\mathcal{R}$ admits MPNE (resp., MPSE) then $\mathcal{R}'$ admits MPNE (resp., MPSE).
\end{thm}

\begin{remark}[Adding the reverse of a reaction or a reversible reaction] 
\label{newdeprev}
From Theorem~\ref{thmnewdepreac}, we could add the reverse of any existing reaction to $\mathcal{R}$ as this would be a dependent reaction. Equally, we could add a new {\em reversible} dependent reaction to $\mathcal{R}$ and preserve MPNE, as this amounts to sequentially adding two new irreversible reactions to $\mathcal{R}$. 
\end{remark}

The next theorem is well known. It appears as Theorem~2 in \cite{craciun2} proved using an approach similar to the one here, and as Lemma~B1 in \cite{banajicraciun2} proved using techniques involving the invariance of Brouwer degree, and closely related to the proof of Proposition~1 in \cite{soule}. In the case of MA kinetics or positive general kinetics adding inflows and outflows of all species, as in Theorem~\ref{thmopenextension}, is equivalent to taking the fully open extension of a CRN, and so the theorem may be stated simply as ``taking fully open extensions preserves the capacity for MPNE and MPSE''.

\begin{thm}[Adding inflows and outflows of all species]
\label{thmopenextension}
Suppose we create $\mathcal{R}'$ from $\mathcal{R}$, by adding to $\mathcal{R}$ the reactions $0 \rightleftharpoons X_i$ for $i$ in $1, \ldots, n$. If $\mathcal{R}$ admits MPNE (resp., MPSE) then $\mathcal{R}'$ admits MPNE (resp., MPSE).
\end{thm}

The following theorem is very easy to prove, and yet surprisingly important in applications when combined with other results in this paper.

\begin{thm}[Adding a trivial species]
\label{thmtrivial}
Suppose we create $\mathcal{R}'$ from $\mathcal{R}$, by adding into some reactions of $\mathcal{R}$ a new species $Y$ which occurs with the same stoichiometry on both sides of each reaction in which it participates. If $\mathcal{R}$ admits MPNE (resp., MPSE) then $\mathcal{R}'$ admits MPNE (resp., MPSE).
\end{thm}

The following theorem is closely related to Theorem~4.2 in \cite{joshishiu}:
\begin{thm}[Adding a new species with inflow and outflow]
\label{thmnewwithopen}
Suppose we create $\mathcal{R}'$ from $\mathcal{R}$, by adding into some reactions of $\mathcal{R}$ the new species $Y$ in an arbitrary way, while also adding the new reaction $0 \rightleftharpoons Y$. If $\mathcal{R}$ admits MPNE (resp., MPSE) then $\mathcal{R}'$ admits MPNE (resp., MPSE).
\end{thm}

Theorem~\ref{thmnewwithopen}, combined with Theorem~\ref{thmnewdepreac} allows us to deduce an important corollary whose claim about MPNE appears, although with different terminology, in \cite{joshishiu}.
\begin{cor}[Inheritance of MPNE and MPSE in the induced subnetwork partial order]
\label{coropeninduced}
Let $\mathcal{R}$ and $\mathcal{R}'$ be fully open CRNs with $\mathcal{R}$ an induced subnetwork of $\mathcal{R}'$. If $\mathcal{R}$ admits MPNE (resp., MPSE) then $\mathcal{R}'$ admits MPNE (resp., MPSE).
\end{cor}
\begin{proof}
We can construct $PN(\mathcal{R}')$ from $PN(\mathcal{R})$ via a sequence of steps as follows: (i) first, for each absent species (if any) we add the species vertex, its inflow and outflow reaction vertices, and all missing arcs (corresponding to Theorem~\ref{thmnewwithopen}); (ii) then, for each remaining absent reaction (if any) we add the reaction vertex and all its incident arcs (corresponding to Theorem~\ref{thmnewdepreac} because a fully open CRN has stoichiometric subspace which is the whole ambient space, and hence each added reaction is now a dependent reaction). If $\mathcal{R}$ admits MPNE (resp., MPSE), then since each step preserves this property, $\mathcal{R}'$ admits MPNE (resp., MPSE).
\end{proof}

\begin{remark}[Notes on Corollary~\ref{coropeninduced}]
The claim about MPNE in Corollary~\ref{coropeninduced} appears as Corollary~4.6 of \cite{joshishiu}. The claim about MPSE is part of Theorem~4.5 in the later work \cite{Joshi.2015aa}. However, it appears that our approach allows considerably weaker kinetic assumptions than needed for Corollary~4.6 of \cite{joshishiu}. Joshi (Theorem 4.13 of \cite{Joshi.2013aa}) showed that the fully open CRNs admitting MPNE, minimal w.r.t. the induced subnetwork partial order, and with only one non-flow reaction are the fully open extensions of either $a_1X\to a_2X$ or $X+Y\to b_1X+b_2Y$,  where $a_2 >a_1 >1$,  respectively $b_1>1$ and $b_2 >1$. Corollary~\ref{coropeninduced} can be improved using Theorem~\ref{thmintermediates} below as indicated in Remark~\ref{remFObest}.

\end{remark}

The next two theorems are somewhat harder than Theorems~\ref{thmnewdepreac}~to~\ref{thmnewwithopen}, and can be regarded as the main results of this paper. 

\begin{thm}[Adding new reversible reactions involving new species]
\label{thmblockadd}
Suppose we create $\mathcal{R}'$ from $\mathcal{R}$, by adding $m \geq 1$ new reversible reactions involving $k$ new species with the following nondegeneracy condition: the submatrix of the stoichiometric matrix of $\mathcal{R}'$ corresponding to the new species has rank $m$ (this implies $k \geq m$). If $\mathcal{R}$ admits MPNE (resp., MPSE) then $\mathcal{R}'$ admits MPNE (resp., MPSE).
\end{thm}

\begin{thm}[Adding intermediate complexes involving new species]
\label{thmintermediates}
Let $Y$ be a list of $k$ new species, and suppose we create $\mathcal{R}'$ from $\mathcal{R}$ by replacing each of the  $m$ reactions:
\[
a_i \cdot X \rightarrow b_i \cdot X \mbox{with a chain}\quad a_i \cdot X \rightarrow c_i \cdot X + \beta_i\cdot Y \rightarrow b_i \cdot X,\,\,(i=1,\ldots,m)\,.\]
Suppose further that the new species $Y$ enter nondegenerately into $\mathcal{R}'$ in the sense that $\beta : = (\beta_1|\beta_2|\cdots|\beta_m)$ has rank $m$ (this implies $k \geq m$). $a_i$, $b_i$ and $c_i$ are arbitrary nonnegative vectors and any or all may coincide. If $\mathcal{R}$ admits MPNE (resp., MPSE) then $\mathcal{R}'$ admits MPNE (resp., MPSE).
\end{thm}

\begin{remark}[Scope of Theorem~\ref{thmintermediates}]
Since we do not impose the assumptions in Remark~\ref{remforbid} we may always formally write a single reaction as multiple reactions; for example we way consider the single reaction $a \cdot X \rightarrow b\cdot X$ as a set of $m$ such reactions, each with rate $\frac{1}{m}$ times the original rate. This clearly does not affect the associated dynamics and can be done while remaining in any reasonable class of kinetics (in particular mass action kinetics, where each new reaction has rate constant which is the original rate constant divided by $m$). Thus in Theorem~\ref{thmintermediates} some reactions may be split multiple times and acquire multiple intermediates. Many (but not all) instances of Theorem~\ref{thmblockadd} also follow from Theorem~\ref{thmintermediates}. For example, adding to a CRN a trivial reaction with the same complex on both sides (with any kinetics) does not alter the vector field. Given a CRN with MPNE, we could add and then split such a trivial reaction using Theorem~\ref{thmintermediates} to create a new reversible reaction involving some new species. The resulting CRN admits MPNE by Theorem~\ref{thmintermediates}. 
\end{remark}

\begin{remark}[Relationship between Theorem~\ref{thmintermediates} and results in \cite{feliuwiufInterface2013}]
Theorem~\ref{thmintermediates} appears to generalise certain aspects of results in \cite{feliuwiufInterface2013}, while being unable to reproduce others. In particular, we allow the introduction of complexes (possibly involving existing species) rather than just new individual species as intermediates; however we forbid the introduction of intermediates in a degenerate way allowed in \cite{feliuwiufInterface2013} (although see the remarks in the conclusions). This divergence is not surprising as the techniques in \cite{feliuwiufInterface2013} are global and algebraic, while those here are local and analytical.
\end{remark}

Theorems~\ref{thmnewdepreac}~to~\ref{thmintermediates} collectively define a partial order $\preceq$ on the set of all CRNs as follows: $\mathcal{R}_1 \preceq \mathcal{R}_2$ if and only if $\mathcal{R}_2$ can be obtained from $\mathcal{R}_1$ by a sequence of modifications (possibly empty) as described in Theorems~\ref{thmnewdepreac}~to~\ref{thmintermediates}. Reflexivity and transitivity are trivial; antisymmetry is ensured by the fact that each modification either increases the number of species or the number of reactions or both. So, for example, the CRNs $\mathcal{R}_1$, $\mathcal{R}_5$, $\mathcal{R}_6$, $\mathcal{R}_9$, $\mathcal{R}_{10}$ and $\mathcal{R}_{11}$ in Section~\ref{secextended} are all $\succ$ $\mathcal{R}_0$.

\begin{remark}[Improving Corollary~\ref{coropeninduced}]
\label{remFObest}
Theorem~\ref{thmintermediates} allows an improvement on Corollary~\ref{coropeninduced}. Given a fully open CRN $\mathcal{R}$ admitting MPNE (resp., MPSE), we may split some reactions as in Theorem~\ref{thmintermediates} to get a CRN $\mathcal{R}'$ admitting MPNE (resp., MPSE). $\mathcal{R}'$ will have some new species, say $Y$, but the stoichiometric matrix of $\mathcal{R}'$ has rank equal to the number of species in $\mathcal{R}'$ by the rank condition in Theorem~\ref{thmintermediates}. We may now add inflows and outflows of the species $Y$ to get the fully open extension $\mathcal{R}''$ of $\mathcal{R}'$. By Theorem~\ref{thmnewdepreac}, $\mathcal{R}''$ admits MPNE (resp., MPSE) since the added reactions were dependent reactions. Now $\mathcal{R}''$ may be a minimal CRN admitting MPNE (resp., MPSE) in the induced subnetwork partial order, but by construction is not minimal w.r.t. the partial order $\preceq$ here. For example, according to Theorem~4.13 of \cite{Joshi.2013aa}, the CRN $X \rightleftharpoons 0, \,\,2X\rightarrow 3X$ with MA kinetics admits MPNE and consequently, by Theorems~\ref{thmintermediates}~and~\ref{thmnewdepreac} applied in that order,
\begin{equation}
Y \rightleftharpoons 0 \rightleftharpoons X,  \,\,2X\rightarrow X + Y \rightarrow 3X \tag{\mbox{$\mathcal{R}$}}
\end{equation}
admits MPNE. It is easy to confirm however, either by direct calculation or by applying Theorem~4.13 of \cite{Joshi.2013aa} and Theorem~3.6 of \cite{JoshiShiu2016}, that no fully open induced subnetwork of $\mathcal{R}$ admits MPNE. Thus $\mathcal{R}$ is a minimal fully open CRN admitting MPNE in the induced subnetwork partial order, while it is not minimal w.r.t. the partial order $\preceq$ here.
\end{remark}

The following corollary is of importance in the analysis of biological systems, and involves the operation of adding an enzymatic mechanism to a set of reactions. It will be used a number of times in the analysis of a particular system in Section~\ref{secbioexample}:
\begin{cor}[Adding enzymatic mechanisms]\label{cor:addEnz}
Let $E$ and  $I_1, \ldots, I_m$ be new species, not involved in $\cal R$, and let $c_i \ge 0$ ($i = 1, \ldots, m$). Suppose we create ${\cal R}'$ from ${\cal R}$ by replacing each of the reactions
\[
a_i\cdot X\to b_i\cdot X \quad \mbox{with a chain} \quad c_iE+ a_i\cdot X\rightleftharpoons I_i\to  c_iE+b_i\cdot X,\,\,(i=1,\ldots,m)\,.
\]
If $\cal R$ admits MPNE (resp., MPSE), then ${\cal R}'$ admits MPNE (resp., MPSE).
\end{cor} 
\begin{proof}
We perform successive modifications that preserve the capacity for MPNE and MPSE:
\begin{enumerate}[align=left,leftmargin=*]
\item We add the trivial species E into all the reactions simultaneously (an application of Theorem~\ref{thmtrivial}):
\[
c_iE+ a_i\cdot X\to  c_iE+b_i\cdot X \text{ replaces } a_i\cdot X\to b_i\cdot X,\quad (i=1,\ldots,m).
\]
\item We add the set of intermediate species $I_i$ (an application of Theorem~\ref{thmintermediates}):
\[
c_iE+ a_i\cdot X\to I_i\to  c_iE+b_i\cdot X  \text{ replaces } c_iE+ a_i\cdot X\to  c_iE+b_i\cdot X,\quad (i=1,\ldots,m).
\]
\item We add the dependent reactions $I_i\to c_iE+a_i\cdot X$ ($m$ applications of Theorem~\ref{thmnewdepreac}):
\[
c_iE+ a_i\cdot X\rightleftharpoons I_i\to  c_iE+b_i\cdot X \text{ replaces } c_iE+ a_i\cdot X\to I_i\to  c_iE+b_i\cdot X,\quad (i=1,\ldots,m)
\]
\end{enumerate}
We could see step (2) as $m$ applications of Theorem~\ref{thmintermediates} or a single application of the theorem. In either case, the nondegeneracy condition of Theorem~\ref{thmintermediates} is easily shown to be met. 
\end{proof}

\section{A biologically motivated example}
\label{secbioexample}

Before turning to the proofs of the theorems, we present an example indicating how the results here intersect with published work on biologically important systems. 

The mitogen-activated (MAPK) cascade is an extensively studied network occurring in various biological processes in eukaryotic cells. The model of the MAPK cascade (Figure \ref{fig:HFfeed}(f)) was developed by Huang and Ferrell \cite{Huang.1996aa}, who showed that the cascade may exhibit ultrasensitive behaviour, making it appropriate for switch-like processes like mitogenesis or cell-fate determination. Addition of a negative feedback in the MAPK cascade (Figure \ref{fig:HFfeed}(g)) was demonstrated numerically to cause oscillatory behaviour by Kholodenko \cite{Kholodenko.2000aa}. Qiao et al. \cite{Qiao.2007aa} later showed numerically that the Huang-Ferrell model without the added negative feedback can also exhibit oscillatory behaviour and bistability. Bistability in the MAPK cascade has also been observed numerically and discussed in earlier papers \cite{Ferrell.1998aa, Kholodenko.2000aa}. Here we show how straightforward applications of our results prove the existence of MPNE and in fact MPSE in the MAPK cascade, with or without negative feedback. With mass action, the MAPK cascade with negative feedback (Figure \ref{fig:HFfeed}(g)) involves 25 species and 36 reactions; we start by proving that the much simpler network in Figure \ref{fig:HFfeed}(a) admits MPSE, and then apply a sequence of modifications that preserve MPSE to ultimately transform network (a) into (g) and (f) of Figure \ref{fig:HFfeed}.

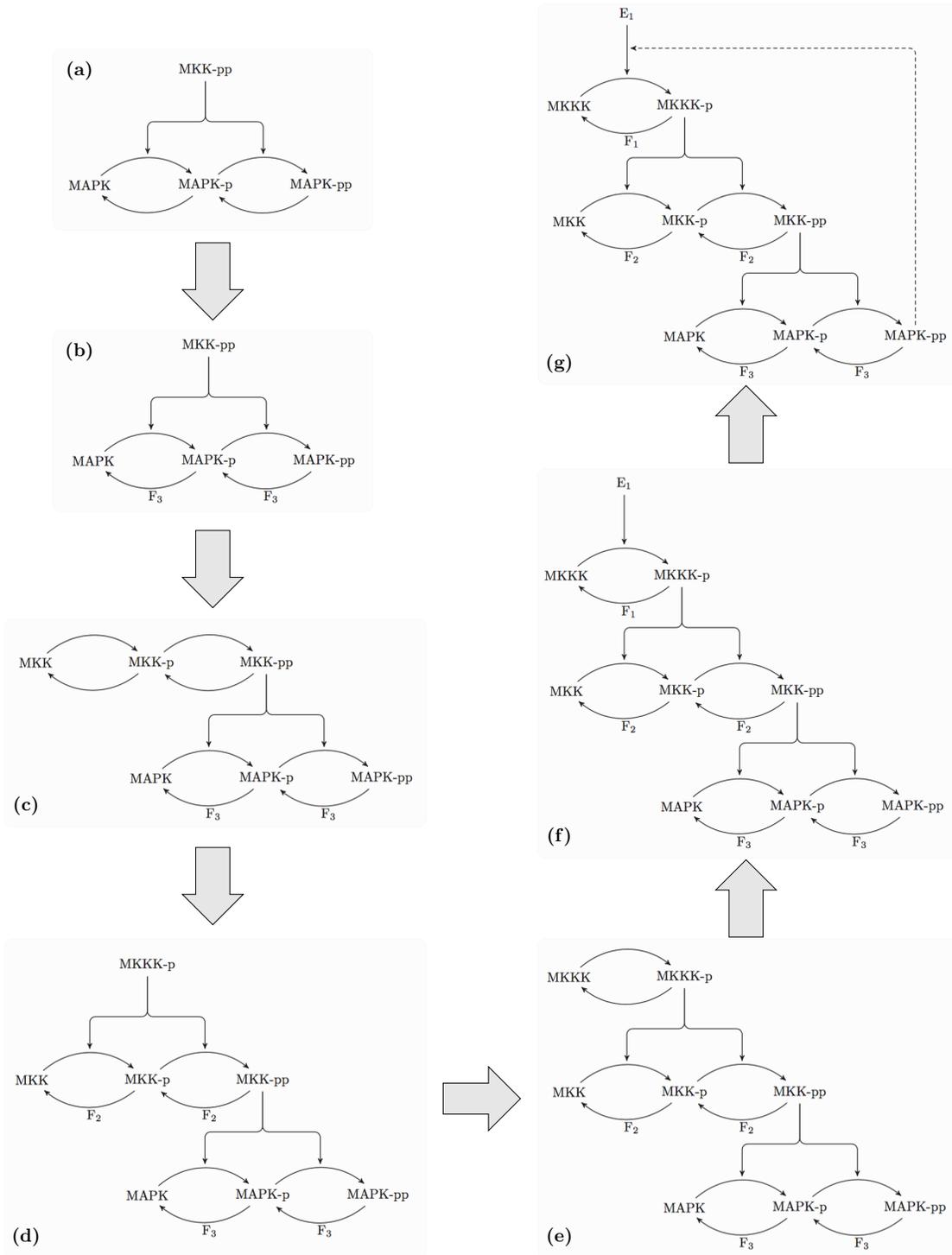
\begin{figure}[!p]
\resizebox{15cm}{!}{
\begin{tikzpicture}
\node (a) at (-3,0) {\includegraphics[scale=.27]{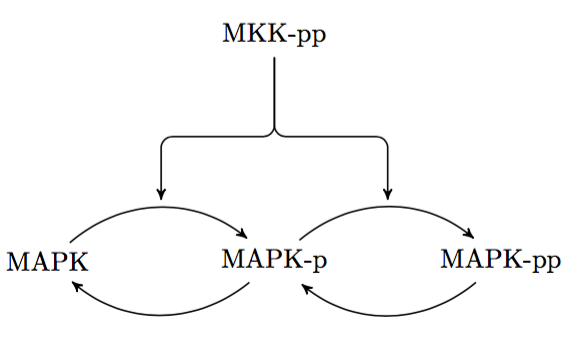}};
\draw [draw=mygray, rounded corners, fill=light-gray, opacity=0.1]
(-6,1.7) -- (-6,-1.7) -- (0,-1.7) -- (0,1.7) -- cycle;
\node () at (-5.5, 1.3) {\bf (a)};

\begin{scope}[yshift=-8]
\node (b) at (-3,-5) {\includegraphics[scale=.27]{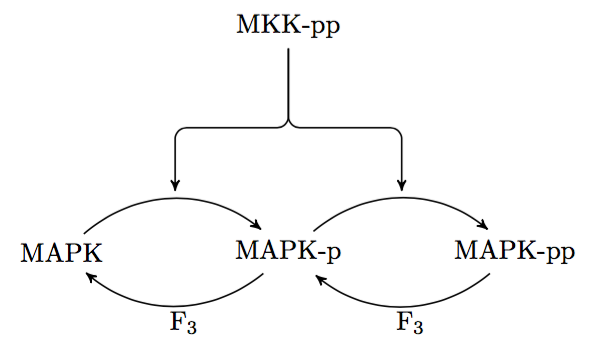}};
\draw [draw=mygray, rounded corners, fill=light-gray, opacity=0.1]
(-6,-3.3) -- (-6,-6.7) -- (0,-6.7) -- (0,-3.3) -- cycle;
\node () at (-5.5, -3.7) {\bf (b)};
\end{scope}

\node (c) at (-3,-11) {\includegraphics[scale=.27]{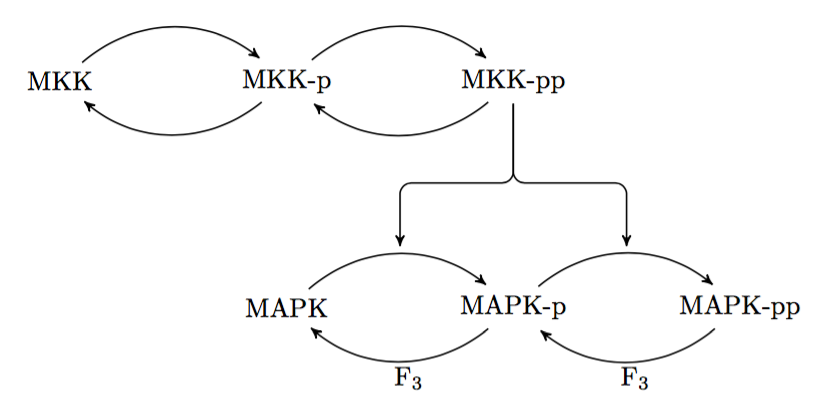}};
\draw [draw=mygray, rounded corners, fill=light-gray, opacity=0.1]
(-6.9,-9) -- (-6.9,-12.9) -- (1,-12.9) -- (1,-9) -- cycle;
\node () at (-6.5, -12.5) {\bf (c)};

\node (d) at (-3,-18) {\includegraphics[scale=.27]{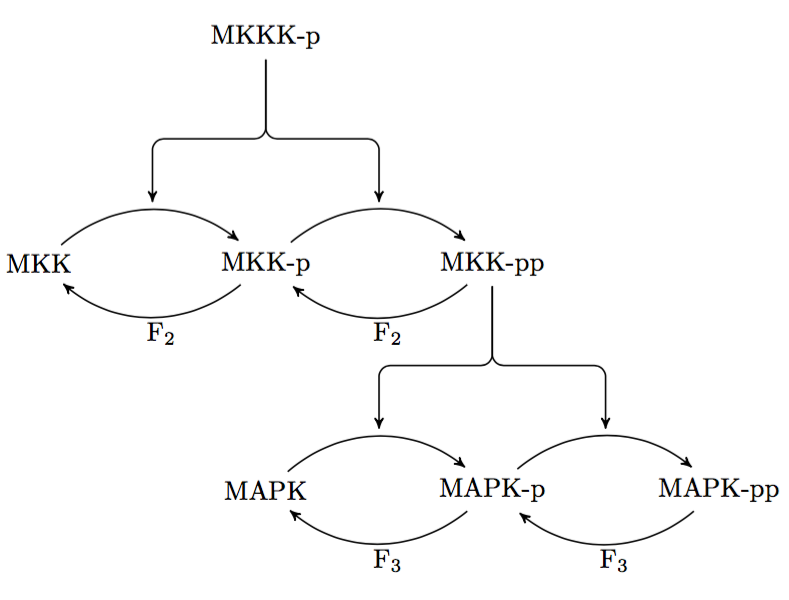}};
\draw [draw=mygray, rounded corners, fill=light-gray, opacity=0.1]
(-6.9,-15) -- (-6.9,-21) -- (1,-21) -- (1,-15) -- cycle;
\node () at (-6.5, -20.6) {\bf (d)};

\node (e) at (7,-18) {\includegraphics[scale=.27]{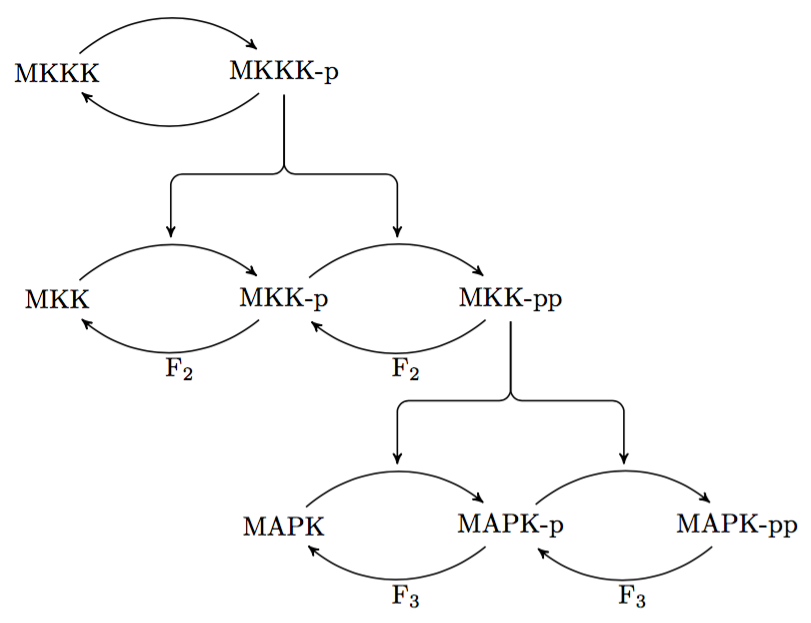}};
\draw [draw=mygray, rounded corners, fill=light-gray, opacity=0.1]
(3.1,-15) -- (3.1,-21) -- (11,-21) -- (11,-15) -- cycle;
\node () at (3.5, -20.6) {\bf (e)};

\begin{scope}[yshift=-8]
\node (f) at (7,-9.5) {\includegraphics[scale=.27]{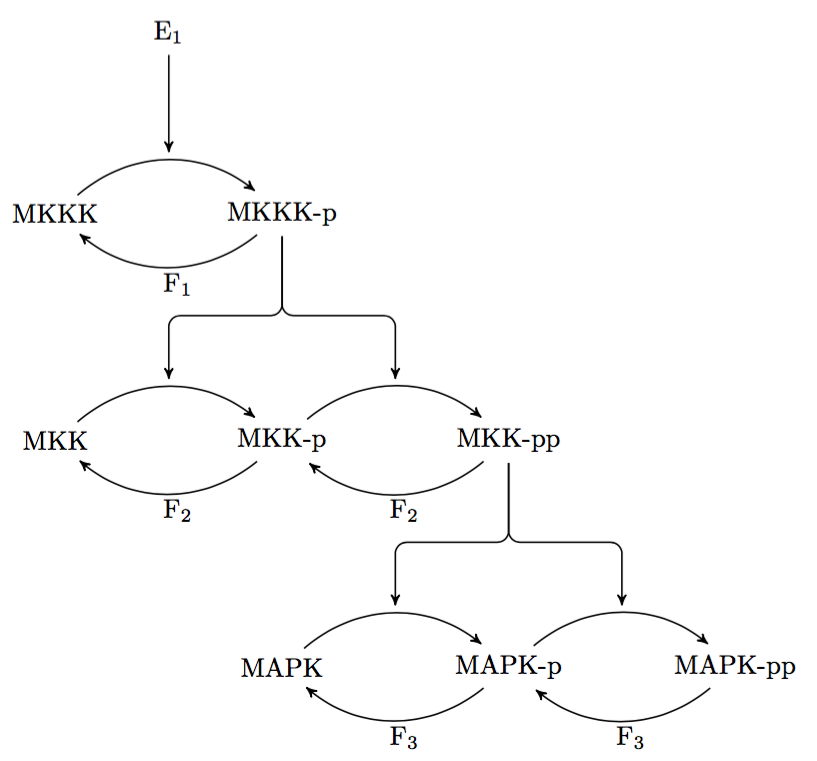}};
\draw [draw=mygray, rounded corners, fill=light-gray, opacity=0.1]
(3.1,-5.9) -- (3.1,-13.2) -- (11,-13.2) -- (11,-5.9) -- cycle;
\node () at (3.5, -12.8) {\bf (f)};
\end{scope}

\node (g) at (7,-1) {\includegraphics[scale=.27]{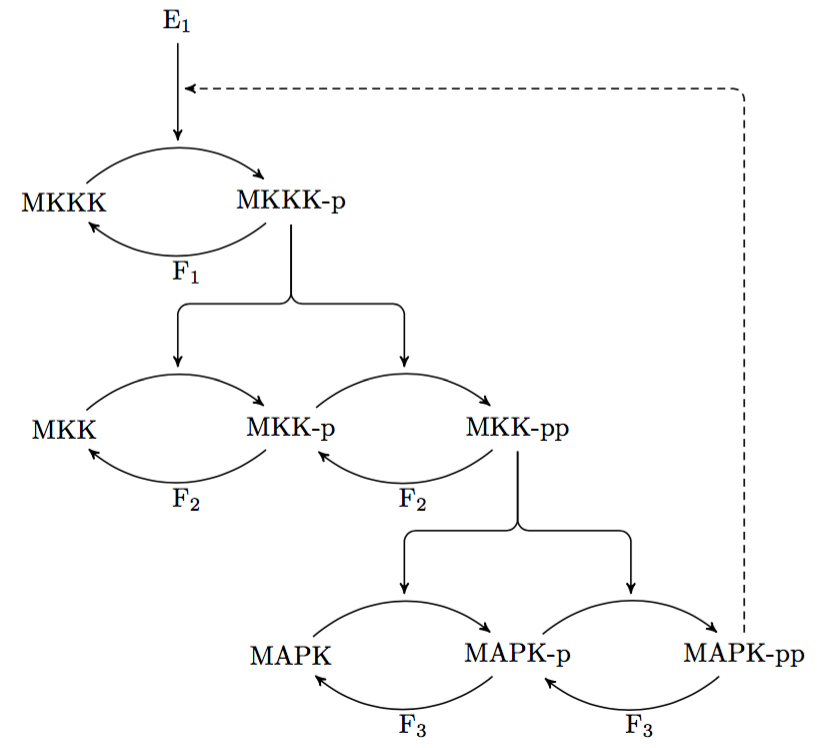}};
\draw [draw=mygray, rounded corners, fill=light-gray, opacity=0.1]
(3.1,2.55) -- (3.1,-4.6) -- (11,-4.6) -- (11,2.55) -- cycle;
\node () at (3.5, -4.2) {\bf (g)};

\begin{scope}[every node/.style={single arrow, draw},
rotate border/.style={shape border uses incircle, shape border rotate=#1}]
\node at (1.9, -18)[minimum height=1.45cm, minimum width = 1.2cm, fill=light-gray]{{\color{light-gray} up}};{up};
\node at (7,-5.5) [minimum height=1.45cm, shape border rotate=90, fill=light-gray]{{\color{light-gray} up}};
\node at (7,-14.4) [minimum height=1.45cm,shape border rotate=90, fill=light-gray]{\color{light-gray}up};

\node at (-3,-13.85) [minimum height=1.45cm,shape border rotate=270, fill=light-gray]{\color{light-gray}up};
\node at (-3,-7.9) [minimum height=1.45cm,shape border rotate=270, fill=light-gray]{\color{light-gray}up};
\node at (-3,-2.5) [minimum height=1.45cm,shape border rotate=270, fill=light-gray]{\color{light-gray}up};
\end{scope}

\end{tikzpicture}
}
\caption{The capacity for MPNE is preserved through successive modifications of the minimal network (a) into the MAPK cascade with negative feedback (g).}\label{fig:HFfeed}
\end{figure}

In what follows, networks (a)--(g) refer to Figure \ref{fig:HFfeed}. We use the following abbreviations in order to aid readability of the reaction networks: X = MAPK, Y = MKK, Z = MKKK.

{\bf (a) A minimal multistationary subnetwork.} The two first phosphorylation steps of X (namely, MAPK) are induced by Y-pp (namely, MKK-pp), but only the first one involves an intermediate species. Both dephosphorylation steps are direct and involve no enzyme. The reactions are as follows:  
\begin{eqnarray}\label{eq:neta}
&\text{Y-pp$\,\,+\,\,$X}\overset{k_1}{\underset{k_2}{\,\,\rightleftharpoons\,\,}} \text{Y-pp--X}\overset{k_3}{\,\,\to\,\,} \text{Y-pp+X-p}\overset{k_4}{\,\,\to\,\,}\text{Y-pp+X-pp}, \qquad \text{X-pp}\overset{k_5}{\,\,\to\,\,} \text{X-p} \overset{k_6}{\,\,\to\,\,} \text{X}
\end{eqnarray}
This network involves 5 species and 6 irreversible reactions. By an exact calculation (Appendix~\ref{appbio}) one can prove that the network with mass action kinetics admits MPNE. This can also be confirmed using CRNToolbox \cite{crntoolbox}. We also confirmed numerically that it admits MPSE. Moreover, it can be checked that the network cannot be constructed from any simpler network which admits MPNE using the modifications described in Theorems~\ref{thmnewdepreac}~to~\ref{thmintermediates}. In this sense the network is a minimal network admitting MPNE with respect to the partial order $\preceq$ and we might refer to it as an ``atom of multistationarity'' borrowing the terminology of \cite{joshishiu}.

{\bf (b) Adding enzymes: the double phosphorylation network.} Network (a) is a simplified version of the much-studied double phosphorylation-dephosphorylation cycle, whose capacity for bistable behaviour has been both demonstrated by numerical simulations, and shown analytically \cite{markevich, Ortega.2006ab, Wang.2008aa}. Its reaction network
\begin{eqnarray}\label{eq:netb} \nonumber
&\text{Y-pp\,\,+\,\,X}\,\,\,\rightleftharpoons\,\,\, \text{Y-pp--X}
\,\,\,\to\,\,\,\text{Y-pp\,\,+\,\,X-p}\,\,\,\rightleftharpoons\,\,\,\text{Y-pp--X-p}\,\,\,\to\,\,\,\text{Y-pp\,\,+\,\,X-pp}\\ 
&\text{F$_3$\,\,+\,\,X-pp}\,\,\,\rightleftharpoons\,\,\, \text{F$_3$--X-pp}\,\,\,\to\,\,\, \text{F$_3$\,\,+\,\,X-p}
\,\,\,\rightleftharpoons\,\,\, \text{F$_3$--X-p}\,\,\,\to\,\,\,\text{F$_3$\,\,+\,\,X}
\end{eqnarray}
is obtained from that of (\ref{eq:neta}) by adding enzymatic mechanisms as in Corollary~\ref{cor:addEnz}: the two reactions $\text{X-pp}\to \text{X-p}$ and $\text{X-p}\to \text{X}$ get replaced by the two chains 
$\text{F$_3+$X-pp}\rightleftharpoons\text{F$_3$--X-pp}\to\text{F$_3+$X-p}$ and 
$\text{F$_3+$X-p}\rightleftharpoons\text{F$_3$--X-p}\to\text{F$_3+$X}$ respectively. It follows that (b), with mass action kinetics, admits MPSE.

{\bf (c) Cascading double phosphorylations I.} A simple version of the MKK double phosphorylation, not mediated by an enzymatic mechanism, is coupled upstream of the double phosphorylation mechanism of MAPK from (b). The network is composed of the reactions in (\ref{eq:netb}), plus
\begin{equation}
\label{eq:netc}
\text{Y}\,\,\rightleftharpoons\,\, \text{Y-p} \,\,\rightleftharpoons\,\, \text{Y-pp}\,.
\end{equation}
Adding these two reversible reactions involving the two new species Y-p (MKK-p)  and Y (MKK) preserves the capacity for MPSE of the new network. This can be regarded as two applications of Theorem~\ref{thmblockadd}, or a single application of the theorem; the nondegeneracy condition is easily seen to be satisfied. Therefore (c), with mass action kinetics, admits MPSE. 

{\bf (d) Cascading double phosphorylations II.} Here the upstream double phosphorylation in (c) is modified to include enzymatic mechanisms. The reaction network of (d) is obtained from that of (c) by replacing the four irreversible reactions (\ref{eq:netc}) by the four chains
\[
\begin{array}{cc}
\text{Z-p$\,\,+\,\,$Y}\,\,\rightleftharpoons\,\, \text{Z-p--Y}\,\,\to\,\, \text{Z-p$\,\,+\,\,$Y-p,}&\text{Z-p$\,\,+\,\,$Y-p}\,\,\rightleftharpoons\,\, \text{Z-p--Y-p}\,\,\to\,\, \text{Z-p$\,\,+\,\,$Y-pp,}\\
\text{F$_2$$\,\,+\,\,$Y-pp}\,\,\rightleftharpoons\,\, \text{F$_2$--Y-pp}\,\,\to\,\, \text{F$_2$$\,\,+\,\,$Y-p,} & \text{F$_2$$\,\,+\,\,$Y-p}\,\,\rightleftharpoons\,\, \text{F$_2$--Y-p}\,\,\to\,\, \text{F$_2$$\,\,+\,\,$Y}.
\end{array}
\] 
Corollary~\ref{cor:addEnz} guarantees that (d) admits MPSE. 

{\bf (e) A three layer cascade.} In network (e) a third layer is added upstream of cascade (d). The network is composed of all reactions of (d) and 
\begin{equation}\label{eq:nete}
\text{Z}\,\,\rightleftharpoons\,\, \text{Z-p}
\end{equation}
Adding this reversible reaction involving the new species Z (namely, MKKK) falls immediately under the scope of Theorem~\ref{thmblockadd}. Therefore (e), with mass action kinetics, admits MPSE.   

{\bf (f) MAPK cascade.} Now we modify (e) by adding enzymatic mechanisms (Corollary~\ref{cor:addEnz}) to the reversible reaction (\ref{eq:nete}), which is replaced by the two chains 
$\text{E$_1$$+$Z}\rightleftharpoons \text{E$_1$--Z}\to \text{E$_1$$+$Z-p}$ and 
$\text{F$_1$$+$Z-p}\rightleftharpoons \text{F$_1$--Z-p}\to \text{F$_1$$+$Z}.$ It follows that the new network (Huang-Ferrell MAPK cascade \cite{Huang.1996aa}, reactions below), with mass action kinetics, admits MPSE:
\begin{equation}
\label{eq:netf}
\begin{array}{c}
\text{E$_1$$\,\,+\,\,$Z}\,\,\rightleftharpoons\,\, \text{E$_1$--Z}\,\,\to\,\, \text{E$_1$$\,\,+\,\,$Z-p},\,\,\,\text{F$_1$$\,\,+\,\,$Z-p}\,\,\rightleftharpoons\,\, \text{F$_1$--Z-p}\,\,\to\,\, \text{F$_1$$\,\,+\,\,$Z}\\
\text{Z-p$\,\,+\,\,$Y}\,\,\rightleftharpoons\,\, \text{Z-p--Y}\,\,\to\,\, \text{Z-p$\,\,+\,\,$Y-p}
\,\,\rightleftharpoons\,\, \text{Z-p--Y-p}\,\,\to\,\, \text{Z-p$\,\,+\,\,$Y-pp}
\\  
 \text{F$_2$$\,\,+\,\,$Y-pp}\,\,\rightleftharpoons\,\, \text{F$_2$--Y-pp}\,\,\to\,\, \text{F$_2$$\,\,+\,\,$Y-p}
\,\,\rightleftharpoons\,\, \text{F$_2$--Y-p}\,\,\to\,\, \text{F$_2$$\,\,+\,\,$Y} \\ 
\text{Y-pp$\,\,+\,\,$X}\,\,\rightleftharpoons\,\, \text{Y-pp--X}
\,\,\to\,\,\text{Y-pp$\,\,+\,\,$X-p}\,\,\rightleftharpoons\,\,\text{Y-pp--X-p}\,\,\to\,\,\text{Y-pp$\,\,+\,\,$X-pp}\\ 
\text{F$_3$$\,\,+\,\,$X-pp}\,\,\rightleftharpoons\,\, \text{F$_3$--X-pp}\,\,\to\,\, \text{F$_3$$\,\,+\,\,$X-p}
\,\,\rightleftharpoons\,\, \text{F$_3$--X-p}\,\,\to\,\,\text{F$_3$$\,\,+\,\,$X}
\end{array} 
\end{equation}

{\bf (g) MAPK cascade with negative feedback.} Finally, we add a negative feedback loop from X-pp as in \cite{Kholodenko.2000aa}. In that paper, the model describes the feedback as noncompetitive inhibition  of Z phosphorylation, whose mechanism can be unpacked as (see for example \cite{Marangoni.2003aa}):
\begin{eqnarray}\label{eq:netg} \label{eq:nf1}
&\text{E}_1\,\,+\,\,\text{X-pp}\,\,\rightleftharpoons\,\, \text{E}_1\text{--X-pp}\\  \label{eq:nf2}
&\text{E}_1\text{--X-pp$\,\,+\,\,$Z}\,\,\rightleftharpoons\,\, \text{E}_1\text{--X-pp--Z}\\ \label{eq:nf3}
&\text{E}_1\text{--X-pp--Z}\,\,\rightleftharpoons\,\, \text{E}_1\text{--Z$\,\,+\,\,$X-pp}.
\end{eqnarray}
These reactions, together with those of (\ref{eq:netf}) make up network (g).
First we add the reversible reactions (\ref{eq:nf1}) and (\ref{eq:nf2}), involving the new species E$_1$--X-pp and E$_1$--X-pp--Z, to the network (\ref{eq:netf}). The nondegeneracy condition is straightforward to check, and by Theorem~\ref{thmblockadd} the resulting network ${\cal R}_1$ admits MPSE. Next, let $w_1, w_2, w_3$ and $w_4$ denote the reaction vectors of (\ref{eq:nf1}), (\ref{eq:nf2}), (\ref{eq:nf3}) and E$_1$+Z $\rightleftharpoons$ $\text{E}_1\text{--Z}$, and note that $w_3=-w_1-w_2+w_4,$ making  (\ref{eq:nf3}) a dependent reaction for ${\cal R}_1$.  By Theorem~\ref{thmnewdepreac} (and Remark~\ref{newdeprev}), we may add (\ref{eq:nf3}) to ${\cal R}_1$, and thus the resulting network (g), with mass action kinetics, admits MPSE.

\section{Proofs of the theorems}
\label{secproofs}
Some additional notation simplifies the presentation. 

\begin{notation}[Entrywise product and entrywise functions]
Given two matrices $A$ and $B$ with the same dimensions, $A \circ B$ will refer to the entrywise (or Hadamard) product of $A$ and $B$, namely $(A\circ B)_{ij} = A_{ij}B_{ij}$. When we apply functions such as $\ln(\cdot)$ and $\exp(\cdot)$ with a vector or matrix argument, we mean entrywise application. Similarly, if $x = (x_1, \ldots, x_n)^\mathrm{t}$ and $y = (y_1, \ldots, y_n)^\mathrm{t}$, then $x/y$ means $(x_1/y_1, x_2/y_2,\ldots, x_n/y_n)^\mathrm{t}$.
\end{notation}

\begin{notation}[Set of integers, vector of ones]
We write $\langle n \rangle$ for $\{1, \ldots, n\}$. We write $\mathbf{1}$ for a vector of ones whose length is inferred from the context, and given some parameter $\epsilon$, we write $\bm{\epsilon}$ for $\epsilon \mathbf{1}$.
\end{notation}

\begin{notation}[Monomials, vector of monomials]
Given $x=(x_1,\ldots, x_n)^\mathrm{t}$ and $a = (a_1, \ldots, a_n)$, $x^a$ is an abbreviation for the monomial $\prod_ix_i^{a_i}$. If $A$ is an $m \times n$ matrix with rows $A_1, \ldots, A_m$, then $x^A$ means the vector of monomials $(x^{A_1}, x^{A_2}, \ldots, x^{A_m})^\mathrm{t}$. 
\end{notation} 

The following three examples demonstrate how entrywise and monomial notation greatly abbreviate otherwise lengthy calculations.
\begin{example}[Rules of exponentiation]
Let $x \in \mathbb{R}^m_{\gg 0}$, $A,B \in \mathbb{R}^{n \times m}$ and $C \in \mathbb{R}^{k \times n}$. Let $O$ refer to the $n \times m$ matrix of zeros. Then (i) $x^O = \mathbf{1}$, (ii) $x^{A+B} = x^A \circ x^B$, and (iii) $(x^A)^C = x^{CA}$. 
\end{example}

\begin{example}[Logarithm of monomials]
Suppose $x \in \mathbb{R}^m_{\gg 0}$, $y,z \in \mathbb{R}^n_{\gg 0}$ and $A,B \in \mathbb{R}^{m \times n}$. If $w=x \circ y^A \circ z^B$, then $\ln w = \ln x + A\ln y + B \ln z$. 
\end{example}

\begin{example}[Differentiation of monomials]
Suppose $k \in \mathbb{R}^m_{\gg 0}$, $x \in \mathbb{R}^n_{\gg 0}$, $A \in \mathbb{R}^{m \times n}$. Let $w \colon \mathbb{R}^n_{\gg 0} \to \mathbb{R}^m_{\gg 0}$ be defined by $w(x) := k \circ x^A$. Then $D_xw = \mathrm{diag}(w)\,A\,\mathrm{diag}(\mathbf{1}/x)$. This expression is familiar as the Jacobian matrix of a rate function of a mass action CRN (\cite{banajiSIAM} for example.)
\end{example}

\begin{notation}[Submatrices and minors of a matrix]
Given an $n \times m$ matrix $A$ and (nonempty) sets $\alpha \subseteq \langle n \rangle$ and $\beta \subseteq \langle m \rangle$, define $A(\alpha|\beta)$ to be the submatrix of $A$ with rows from $\alpha$ and columns from $\beta$. If $|\alpha| = |\beta|$, then $A[\alpha|\beta] := \mathrm{det}(A(\alpha|\beta))$. $A(\alpha)$ is shorthand for $A(\alpha|\alpha)$, and $A[\alpha]$ is shorthand for the principal minor $A[\alpha|\alpha]$.
\end{notation}

For our purposes here, the implicit function theorem (IFT) takes the form:
\begin{lemma1}[Implicit Function Theorem]
\label{lemIFT}
Let $X \subseteq \mathbb{R}^n \times \mathbb{R}^m \simeq \mathbb{R}^{n+m}$ be open and $F\colon X \to \mathbb{R}^n$ be $C^1$. Suppose $F(a,b)=0$ for some $(a,b) \in X$ and the Jacobian matrix $D_1F(a,b)$ (namely with respect to the first variables) is nonsingular. Then there exist $U \subseteq \mathbb{R}^m$, $V \subseteq \mathbb{R}^n$ both open with $(a,b) \in W:=V\times U \subseteq X$, and a $C^1$ function $\phi \colon U \to V$ satisfying $\phi(b) = a$, and 
\[
\{(x,y) \in W\,|\, F(x,y)=0\} = \{(\phi(y),y)\,|\, y \in U\}\,,
\]
namely the zero set of $F$ in $W$ is precisely the graph of $\phi$.
\end{lemma1}
\begin{proof}
See, for example, Theorem 9.28 in \cite{rudin}.
\end{proof}

Application of the IFT will frequently take the following form:
\begin{cor}[The IFT and reduced Jacobian matrices]\label{corift}Let $X\subseteq \mathbb{R}^n$ be open and let $S \subseteq \mathbb{R}^n$ be a $k$-dimensional linear subspace of $\mathbb{R}^n$. Let $\epsilon_1>0$ and consider a $C^1$ function $F\colon X \times (-\epsilon_1,\epsilon_1) \to \mathbb{R}^n$ s.t. for each fixed $\epsilon$, $F_\epsilon := F(\cdot, \epsilon)$ satisfies $\mathrm{im}\,F_\epsilon \subseteq S$. Let $p \in X$ be a nondegenerate zero of $F_0$ w.r.t. $S$, namely, $F_0(p) = 0$ and $J_SF_0(p) \neq 0$. Then there exists $0< \epsilon_0\leq\epsilon_1$ and a $C^1$ curve $\phi\colon(-\epsilon_0,\epsilon_0) \to p+S$ with $\lim_{\epsilon\to 0} \phi(\epsilon) = p$ and $F_\epsilon(\phi(\epsilon)) = 0$ for $\epsilon \in (-\epsilon_0, \epsilon_0)$. Moreover, $\phi(\epsilon)$ is an isolated zero of $F_\epsilon$ on $p+S$.
\end{cor}
\begin{proof}
As in Definition~\ref{notationreddet}, let $M$ be a matrix whose columns are a basis for $S$, and define $\hat{F}_\epsilon\colon X \to \mathbb{R}^k$ via $F_\epsilon(x) = M\hat{F}_\epsilon(x)$. Let $Y = \{y \in \mathbb{R}^k \,|\, p+My \in X\}$. Note that $Y$ is open and includes the origin. Define $G \colon Y\times (-\epsilon_1, \epsilon_1) \to \mathbb{R}^k$ by $G(y,\epsilon) = \hat{F}_\epsilon(p+My)$. Note that (fixing $\epsilon$) $D_MF_\epsilon(x) = (D\hat{F}_\epsilon(x))M = D_yG(y, \epsilon)$ (where $x = p+My$), namely the reduced Jacobian matrix of $F_\epsilon$ w.r.t. $M$ is just the Jacobian matrix of $G$ w.r.t. its first argument. As $p$ is a nondegenerate zero of $F_0$, we have $G(0,0) = F_0(p) =0$ and $D_yG(0,0) = D_MF_0(p)$ is nonsingular. We may then apply the IFT to $G$ giving an open neighbourhood $V$ of $0$ in $\mathbb{R}^k$, some $\epsilon_0>0$, and a function $\hat\phi\colon (-\epsilon_0,\epsilon_0) \to \mathbb{R}^k$ s.t.
\[
\{(y,\epsilon) \in V \times (-\epsilon_0,\epsilon_0)\,|\, G(y,\epsilon)=0\} = \{(\hat\phi(\epsilon),\epsilon)\,|\, \epsilon \in (-\epsilon_0,\epsilon_0)\}\,.
\] 
The result follows with $\phi = p+M\hat\phi$. 
\end{proof}

\begin{lemma1}[Matrices with Hurwitz diagonal blocks]
\label{lemHurwitz}
Let $\epsilon \in (0, \epsilon_0)$ for some $\epsilon_0 > 0$ and suppose we have an $\epsilon$-dependent family of matrices partitioned into four blocks
\[
M(\epsilon) = \left(\begin{array}{cc}A(\epsilon)&B(\epsilon)\\C(\epsilon)&D(\epsilon)\end{array}\right)\,,
\]
with $A$ and $D$ square. Suppose further that: (i) $A_0:=\lim_{\epsilon \to 0+}A(\epsilon)$ is defined and Hurwitz; (ii) $\lim_{\epsilon \to 0+}\epsilon B(\epsilon)$ and $\lim_{\epsilon \to 0+}\epsilon C(\epsilon)$ are zero matrices; (iii) $D_0 := \lim_{\epsilon \to 0+}\epsilon D(\epsilon)$ is defined and Hurwitz. Then there exists $\epsilon' > 0$ s.t. $M(\epsilon)$ is Hurwitz for all $\epsilon \in (0, \epsilon')$. 
\end{lemma1}
\begin{proof}
  Let $A$ be $n \times n$ and $D$ be $m \times m$. We calculate
\begin{equation}
\label{eqspeca}
\lim_{\epsilon \to 0+}\mathrm{det}(\epsilon M(\epsilon)-\lambda I) = \left|\begin{array}{cc}-\lambda I&0\\0&D_0-\lambda I\end{array}\right| = (-\lambda)^n\mathrm{det}(D_0-\lambda I)\,.
\end{equation}
By the continuous dependence of eigenvalues of a matrix on its entries, (\ref{eqspeca}) tells us that $m$ eigenvalues of $\epsilon M(\epsilon)$ approach the eigenvalues of $D_0$, and so there exist $\theta > 0$ and $\epsilon_1>0$ s.t. for $0 < \epsilon < \epsilon_1$ $m$ eigenvalues of $\epsilon M(\epsilon)$ lie in $\mathbb{C}_-\backslash B_\theta$, where $B_\theta := \{z\in\mathbb{C}\,\colon\, |z| \leq \theta\}$.

On the other hand, expanding $\mathrm{det}(\epsilon M(\epsilon) - \epsilon \lambda I)$ along the top $n$ rows gives:
\begin{equation}
\label{eqspecb}
\epsilon^{n+m}\mathrm{det}(M(\epsilon)-\lambda I) = \mathrm{det}(\epsilon M(\epsilon) - \epsilon \lambda I) = \epsilon^nG_\epsilon(\lambda)\,,
\end{equation}
where $G_\epsilon(\lambda) := \mathrm{det}(\epsilon D(\epsilon))\mathrm{det}\left(A(\epsilon)-\lambda I\right)+O(\epsilon)$. Namely, eigenvalues of $M(\epsilon)$ are the zeros of $G_\epsilon(\lambda)$. But $G_0(\lambda):=\lim_{\epsilon \to 0+}G_\epsilon(\lambda) = \mathrm{det}D_0\,\mathrm{det}\left(A_0-\lambda I\right)$, so (since $\mathrm{det}\,D_0 \neq 0$) the zeros of $G_0$ are the eigenvalues of $A_0$. Consider any contour $C$ in $\mathbb{C}_-$ with the spectrum of $A_0$ in its interior. Choose $\epsilon_2>0$ such that $0<\epsilon < \epsilon_2$ implies that $|G_\epsilon - G_0|<|G_0|$ on $C$, and $\sup_{z \in \mathrm{int}\,C}\epsilon |z| < \theta$. Then by Rouch\'e's theorem (Thm. 7.7 in \cite{Priestley} for example), $G_\epsilon$ has the same number of zeros in $C$ as $G_0$. Thus (\ref{eqspecb}) tells us that, for $0<\epsilon < \epsilon_2$, $n$ eigenvalues of $M(\epsilon)$ lie in $C$ and hence in $\mathbb{C}_-$ (in fact, it is easy to refine the argument to see that they approach the spectrum of $A_0$ as $\epsilon \to 0+$). The corresponding $n$ eigenvalues of $\epsilon M(\epsilon)$ lie in $\mathbb{C}_- \cap B_\theta$. In summary, for $0 < \epsilon < \epsilon':=\min\{\epsilon_1,\epsilon_2\}$, $\epsilon M(\epsilon)$ has $m$ eigenvalues in $\mathbb{C}_-\backslash B_\theta$ and $n$ eigenvalues in $\mathbb{C}_- \cap B_\theta$, and so is Hurwitz stable. The same clearly holds for $M(\epsilon)$.
\end{proof}

\subsection{Proofs of Theorems~\ref{thmnewdepreac}-\ref{thmnewwithopen}} In each proof below, the species $X = (X_1, \ldots, X_n)$ have corresponding concentration vector $x = (x_1, \ldots, x_n)^\mathrm{t}$. It is assumed at the outset that $\mathcal{R}$ admits MPNE and the rate vector $v(x) = (v_1(x), v_2(x), \ldots, v_{r_0}(x))^\mathrm{t}$ of $\mathcal{R}$ is fixed such that MPNE occurs. $p$ and $q$ are a pair of positive, compatible, nondegenerate equilibria of $\mathcal{R}$. The stoichiometric matrix of $\mathcal{R}$, $\Gamma$, is an $n \times r_0$ matrix with rank $r$. Columns of the matrix $\Gamma_0$ are a basis for $S:=\mathrm{im}\,\Gamma$, and $Q$ is defined via $\Gamma = \Gamma_0Q$. $S_p:=p+S$ refers to the coset of $S$ containing $p$, and by assumption $S_p = S_q$. Any new species in $\mathcal{R}'$ are termed $Y$ with concentration vector $y$.

\begin{myproof}{Theorem~\ref{thmnewdepreac}}
Let the new reaction be $a\cdot X \rightarrow a' \cdot X$ so the new stoichiometric matrix is $\Gamma' = [\Gamma|\alpha]$ where $\alpha = a'-a$. Define $c$ by $\alpha = \Gamma_0 c$. We give the new reaction mass action kinetics and set its rate constant to be $\epsilon$, so that the evolution of $\mathcal{R}'$ is governed by:
\[
\dot x = \Gamma v(x) + \epsilon \alpha x^a =: F(x;\epsilon)\,.
\]
The reduced Jacobian matrix of $F$ w.r.t. $\Gamma_0$ is
\[
D_{\Gamma_0}F(x;\epsilon) = \left(Q\,\,\,c\right)\left(\begin{array}{c}Dv(x)\\\epsilon x^a a^\mathrm{t}\mathrm{diag}(\mathbf{1}/x)\end{array}\right)\Gamma_0 = QDv(x)\Gamma_0 + \epsilon x^a ca^\mathrm{t}\mathrm{diag}(\mathbf{1}/x) \Gamma_0\,. 
\]
So $D_{\Gamma_0}F(p;0) = QDv(p)\Gamma_0$ is nonsingular as $p$ is a nondegenerate equilibrium of $\mathcal{R}$. Consequently, by Corollary~\ref{corift}, there exists, for each sufficiently small $\epsilon > 0$, an equilibrium $p_\epsilon$ on $S_p$ and such that $\lim_{\epsilon \to 0+}p_\epsilon = p$. By continuity of $D_{\Gamma_0}F$, $p_\epsilon$ is nondegenerate for sufficiently small $\epsilon > 0$, and linearly stable for $\mathcal{R}'$ if $p$ was linearly stable for $\mathcal{R}$. The same arguments apply to give a nondegenerate equilibrium $q_\epsilon$ on $S_p$ and such that $\lim_{\epsilon \to 0+}q_\epsilon = q$. As $p_\epsilon$ and $p_\epsilon$ are distinct for small $\epsilon>0$, $\mathcal{R}'$ displays MPNE, and MPSE if $p$ and $q$ were linearly stable. 
\end{myproof}

\begin{remark}[Notes on Theorem~\ref{thmnewdepreac}]
\label{remnewdepreac}
It is clear from the proof that the added reaction in Theorem~\ref{thmnewdepreac} can be given kinetics in any class provided the rate is a $C^1$ function of the concentrations and can be made arbitrarily small on any compact set in $\mathbb{R}^n_{\gg 0}$. The theorem generalises to other structurally stable objects, notably nondegenerate periodic orbits \cite{banajiCRNosci}. 
\end{remark}

\begin{myproof}{Theorem~\ref{thmopenextension}}
The stoichiometric matrix of $\mathcal{R}'$ can be taken to be $\Gamma' = [\Gamma|I]$ where $I$ is the $n \times n$ identity matrix. Let $\hat{x}$ be any point on $S_p$, and let the $i$th inflow-outflow reaction (treated as a single reversible reaction) have mass action kinetics with forward and backwards rate constants $\epsilon \hat{x}_i$ and $\epsilon$ respectively, so that $\epsilon(\hat{x}_i-x_i)$ is the rate of the $i$th added reaction. Evolution of $\mathcal{R}'$ is then governed by the ODEs $\dot x = \Gamma v(x) + \epsilon (\hat{x} - x)$. 

Let $\Gamma_0' = [\Gamma_0|T]$ where the columns of $T$ are any basis for $S^\perp$. $\Gamma_0'$ is nonsingular and hence its columns are a basis for $\mathbb{R}^n$. Define $z = (\hat{z}, \doublehat{z}) \in \mathbb{R}^{r} \times \mathbb{R}^{n-r}$ by $x = \hat{x} + \Gamma_0'z$. Define $z_p$ and $z_q$ by $p = \hat{x} + \Gamma_0'z_p$ and $q = \hat{x} + \Gamma_0'z_q$. Note that $\Gamma = \Gamma_0'\left(\begin{array}{c}Q\\0\end{array}\right)$. In this coordinate system equilibria of $\mathcal{R}'$ occur when 
\[
\Gamma v(\hat{x} + \Gamma_0'z) - \epsilon \Gamma_0'z=0\,, \quad \mbox{or equivalently,} \quad F(\hat{z}, \doublehat{z};\epsilon) := \left(\begin{array}{c}Q\\0\end{array}\right) v(\hat{x} + \Gamma_0'z) - \epsilon \left(\begin{array}{c}\hat{z}\\\doublehat{z}\end{array}\right) =0,
\]
(as $\Gamma_0'$ is nonsingular). For $\epsilon > 0$, $F$ has the same zeros as
\[
G(\hat{z}, \doublehat{z};\epsilon):=\left(\begin{array}{c}Q\\0\end{array}\right) v(\hat{x} + \Gamma_0'z) - \left(\begin{array}{c}\epsilon\hat{z}\\\doublehat{z}\end{array}\right)\,.
\]
Moreover $G(z_p;0) = G(z_q;0)=0$ (as $\doublehat{z} = 0$ at $z = z_p$, $\Gamma v(p) = 0 \Leftrightarrow Qv(p) = 0$, and $\Gamma v(q) = 0 \Leftrightarrow Qv(q) = 0$). Differentiating gives
\[
DG(z_p;0) = \left(\begin{array}{cc}QDv(p)\Gamma_0&QDv(p)T\\0&-I\end{array}\right)\,,
\]
which is nonsingular as $QDv(p)\Gamma_0$ is. By the IFT there exists, for sufficiently small $\epsilon > 0$, a $C^1$ curve $\epsilon \mapsto z_{p,\epsilon}$ of zeros of $G$ satisfying $\lim_{\epsilon \to 0+}z_{p,\epsilon} = z_p$. The same argument holds to give a $C^1$ curve $z_{q,\epsilon}$ of zeros of $G$ satisfying $\lim_{\epsilon \to 0+}z_{q,\epsilon} = z_q$. By continuity of the determinant, they are nondegenerate zeros of $G$ for sufficiently small $\epsilon>0$, and the same holds for $F$ as a quick calculation reveals that $\mathrm{det}\,DF(z;\epsilon) = \epsilon^{n-r}\mathrm{det}\,DG(z;\epsilon)$. Since $0 \ll p \neq q \gg 0$, $p_\epsilon := \hat{x} + \Gamma_0'z_{p,\epsilon}$ and $q_\epsilon := \hat{x} + \Gamma_0'z_{q,\epsilon}$ are positive and distinct for sufficiently small $\epsilon$. Thus $\mathcal{R}'$ admits MPNE. The spectrum of $DF(z_{p,\epsilon};\epsilon)$ is precisely $-\epsilon$ with multiplicity $n-r$ with the remaining $r$ eigenvalues approaching the spectrum of $Q Dv(p)\Gamma_0$ as $\epsilon \to 0+$. Thus if $p$ is a linearly stable equilibrium of $\mathcal{R}$, then $p_\epsilon$ is a linearly stable equilibrium of $\mathcal{R}'$ for sufficiently small $\epsilon>0$, and the same holds for $q$. In other words, if $\mathcal{R}$ admits MPSE, then so does $\mathcal{R}'$.
\end{myproof}

\begin{remark}[Geometric interpretation of Theorem~\ref{thmopenextension}]
Inflows and outflows were chosen to guarantee that $S_p$ remained locally invariant for $\mathcal{R}'$ and that for sufficiently small $\epsilon > 0$, the vector field of $\mathcal{R}'$ restricted to $S_p$ was $\epsilon$-close to that of $\mathcal{R}$ restricted to $S_p$ in a natural sense. Nondegeneracy of $p$ and $q$ for $\mathcal{R}$ ensured existence and nondegeneracy w.r.t. $S_p$ of $p_\epsilon$ and $q_\epsilon$ for sufficiently small $\epsilon >0$. Moreover the asymptotic stability of $S_p$ for $\epsilon > 0$ guaranteed nondegeneracy of $p_\epsilon$ and $q_\epsilon$ w.r.t. directions transverse to $S_p$. 
\end{remark}

\begin{myproof}{Theorem~\ref{thmtrivial}}
Let $\mathcal{R}'$ have reaction rates $w(x,y)$. Choose $w$ so that $w_j(x, y) = y^{s_j}v_j(x)$ where $s_j$ is the stoichiometry of $Y$ on the left of reaction $j$. We see that
\[
\Gamma' = \left(\begin{array}{c}\Gamma\\0\end{array}\right) \quad \mbox{and} \quad \Gamma_0' = \left(\begin{array}{c}\Gamma_0\\0\end{array}\right)
\]
are the stoichiometric matrix of $\mathcal{R}'$ and a matrix whose columns are a basis for the stoichiometric subspace of $\mathcal{R}'$ respectively. $\Gamma' = \Gamma_0'Q$ and the evolution of $\mathcal{R}'$ is governed by 
\[
\left(\begin{array}{c}\dot x\\ \dot y\end{array}\right) = \Gamma'w(x,y) =: F(x,y).
\] 
Since $w(p,1) = v(p)$ and $\Gamma v(p) = 0$, $\Gamma'w(p,1) = 0$. So $(p,1)$ is a positive equilibrium of $\mathcal{R}'$, and likewise for $(q,1)$. Moreover these equilibria are compatible as $(p,1) - (q,1) = (p-q,0) \in \mathrm{im}\,\Gamma \times \{0\} = \mathrm{im}\,\Gamma'$. Nondegeneracy is easy to confirm as
\begin{equation}
\label{eqtriv}
D_{\Gamma_0'}F(p,1) = Q\,Dw(p,1)\,\Gamma_0' = Q\,(Dv(p), D_yw(p,1))\left(\begin{array}{c}\Gamma_0\\0\end{array}\right) = QDv(p)\Gamma_0\,.
\end{equation}
Recall that nondegeneracy of $p$ for $\mathcal{R}$ is equivalent to nonsingularity of $QDv(p)\Gamma_0$, and so $(p,1)$ is a nondegenerate equilibrium of $\mathcal{R}'$. A similar argument holds for $(q,1)$, and so $\mathcal{R}'$ displays MPNE. If $p$ is linearly stable w.r.t. $\mathrm{im}\,\Gamma$ then, by (\ref{eqtriv}), $(p,1)$ is linearly stable w.r.t. $\mathrm{im}\,\Gamma'$, and the same holds for $q$ and $(q,1)$. Thus if $\mathcal{R}$ admits MPSE then so does $\mathcal{R}'$.
\end{myproof}

\begin{remark}[Notes on Theorem~\ref{thmtrivial}]
\label{remkintriv}
In Theorem~\ref{thmtrivial}, the vector field of $\mathcal{R}'$ restricted to $\{(x,y): y = 1\}$ is precisely that of $\mathcal{R}$, and so not just MPNE, but all dynamical behaviours of $\mathcal{R}$ are reproduced on this set. This can be achieved with any class of kinetics where we can ensure that $w(x,1) = v(x)$ including, of course, mass action (see the discussion of ``species-extensions'' in \cite{banajiCRNosci}).
\end{remark}

\begin{myproof}{Theorem~\ref{thmnewwithopen}}
Let $s_i$ be the net stoichiometry change of $Y$ in the $i$th reaction and $s := (s_1, s_2, \ldots, s_{r_0})$. Give the new reaction $0 \rightleftharpoons Y$ mass action kinetics with both rate constants set to be $\frac{1}{\epsilon}$. The new rates $w$ for the other reactions are set as in Theorem~\ref{thmtrivial} (consistent with, but not assuming, mass action kinetics for $\mathcal{R}$ and $\mathcal{R}'$), so that $w(x,y)$ satisfies $w(x,1) = v(x)$, and $D_xw(x, 1) = Dv(x)$. Define
\[
\Gamma_0' := \left(\begin{array}{cc}\Gamma_0&0\\0&1\end{array}\right), \quad Q' := \left(\begin{array}{cc}Q&0\\\epsilon s & 1\end{array}\right)\,, \quad Q'_1 := \left(\begin{array}{cc}Q&0\\s & 1\end{array}\right)\,.
\]
$\Gamma' = \Gamma_0'Q'_1$ is the stoichiometric matrix of $\mathcal{R}'$. Let $S':=\mathrm{im}\,\Gamma' (=\mathrm{im}\,\Gamma_0')$. $\mathcal{R}'$ gives rise to the following singularly perturbed system:
\[
\left(\begin{array}{c}\dot x\\\dot y\end{array}\right) = \left(\begin{array}{cc}\Gamma&0\\s&1\end{array}\right)\left(\begin{array}{c} w(x,y)\\\frac{1}{\epsilon}(1-y)\end{array}\right) = \Gamma_0' Q'_1\left(\begin{array}{c} w(x,y)\\\frac{1}{\epsilon}(1-y)\end{array}\right) =:  F(x,y,\epsilon)\,.
\] 
For $\epsilon > 0$, $F(x,y;\epsilon)$ has the same zeros as:
\[
G(x,y;\epsilon):= \left(\begin{array}{cc}I&0\\0&\epsilon\end{array}\right)F(x,y;\epsilon) = \Gamma_0'Q'\left(\begin{array}{c} w(x,y)\\1-y\end{array}\right) = \left(\begin{array}{cc}\Gamma&0\\\epsilon s&1\end{array}\right)\left(\begin{array}{c} w(x,y)\\1-y\end{array}\right)\,.
\]
Moreover $G(p,1;0) = G(q,1;0) = 0$ and we can calculate:
\begin{eqnarray*}
D_{\Gamma_0'}G(x,y,\epsilon) 
&=& \left(\begin{array}{cc}QD_xw(x,y)\Gamma_0&QD_yw(x,y)\\\epsilon s D_xw(x,y)\Gamma_0 & \epsilon s D_yw(x,y)-1\end{array}\right)\,.
\end{eqnarray*}
Using the fact that $D_xw(p,1) = Dv(p)$, we calculate $J_{S'}G(p,1,0) = -\mathrm{det}(QDv(p)\Gamma_0) \neq 0$. Thus 
Corollary~\ref{corift} gives, for sufficiently small $\epsilon > 0$, a $C^1$ curve $p_\epsilon := (x(\epsilon), y(\epsilon))$ of zeros of $G$ on $(p,1)+\mathrm{im}\Gamma'$, such that $\lim_{\epsilon \to 0+}(x(\epsilon), y(\epsilon)) = (p,1)$. Positivity of $p_\epsilon$ for sufficiently small $\epsilon$ is immediate as $(p,1) \gg 0$. Nondegeneracy of $p_\epsilon$ for sufficiently small $\epsilon > 0$ follows from continuity of $J_{S'}G(x(\epsilon),y(\epsilon);\epsilon)$ and the fact that $J_{S'}F(x,y;\epsilon) = \frac{1}{\epsilon}\,J_{S'}G(x,y;\epsilon)$. In a similar way we have, for sufficiently small $\epsilon > 0$, a curve $q_\epsilon$ of nondegenerate positive equilibria of $\mathcal{R}'$ on $(q,1)+S'$, and such that $\lim_{\epsilon \to 0+}q_\epsilon = (q,1)$. As $(p,1)-(q,1) = (p-q,0) \in \mathrm{im}\,\Gamma'$, $p_\epsilon$ and $q_\epsilon$ are compatible equilibria of $\mathcal{R}'$. Thus $\mathcal{R}'$ admits MPNE.

If $p$ is linearly stable for $\mathcal{R}$, then $p_\epsilon$ is linearly stable for $\mathcal{R}'$ for sufficiently small $\epsilon > 0$. Define
\[
M(\epsilon) :=D_{\Gamma_0'}F(x(\epsilon),y(\epsilon),\epsilon) = \left(\begin{array}{cc}QD_xw(x(\epsilon),y(\epsilon))\Gamma_0&QD_yw(x(\epsilon),y(\epsilon))\\s D_xw(x(\epsilon),y(\epsilon))\Gamma_0 & s D_yw(x(\epsilon),y(\epsilon))-\frac{1}{\epsilon}\end{array}\right)\,.
\]
We can check quickly that if $p$ is linearly stable for $\mathcal{R}$, namely $\lim_{\epsilon \to 0+} QD_xw(x(\epsilon),y(\epsilon))\Gamma_0 = QD_xw(p,1)\Gamma_0 = QDv(p)\Gamma_0$ is Hurwitz, then the hypotheses of Lemma~\ref{lemHurwitz} are satisfied, and for sufficiently small $\epsilon >0$ all eigenvalues of $M(\epsilon)$ lie in $\mathbb{C}_{-}$. A similar argument applies to $q_\epsilon$ and so, if $\mathcal{R}$ admits MPSE, then so does $\mathcal{R}'$.
\end{myproof}

\begin{remark}[Kinetic assumptions and extensions in Theorem~\ref{thmnewwithopen}]
\label{remnewwithopen}
The kinetic assumptions in Theorem~\ref{thmnewwithopen} can be greatly weakened as seen in the analogous result on periodic orbits in \cite{banajiCRNosci}.
\end{remark}

\subsection{Proofs of Theorems~\ref{thmblockadd}~and~\ref{thmintermediates}}
We begin with a lemma useful in both proofs:
\begin{lemma1}
\label{lemnondegen}
Let $\Gamma$ be an $n \times r_0$ matrix, $\Gamma_0$ be an $n \times r$ matrix whose columns are a basis for $\mathrm{im}\,\Gamma$, and $Q$ a matrix defined by $\Gamma = \Gamma_0Q$. Let $\alpha$ be an $n \times m$ matrix, and $\beta$ a $k \times m$ matrix with rank $m$. Let $v\colon\mathbb{R}^n_{\gg 0} \to \mathbb{R}^{r_0}$ and $f\colon \mathbb{R}^n_{\gg 0} \times \mathbb{R}^k \times \mathbb{R}\to \mathbb{R}^m$ be some $C^1$ functions. Define $\Gamma'$ and $F\colon \mathbb{R}^n_{\gg 0} \times \mathbb{R}^k \times \mathbb{R}\to \mathbb{R}^n \times \mathbb{R}^k$ by
\[
\Gamma' := \left(\begin{array}{cc}\Gamma&\alpha\\0&\beta\end{array}\right)\,\quad \mbox{and}\quad
F(x, y;\epsilon) := \Gamma'\left(\begin{array}{c}v(x)\\f(x,y;\epsilon)\end{array}\right)\,.
\]
Clearly $\mathrm{im}\,F \subseteq \mathrm{im}\,\Gamma'$. Suppose there is a $C^1$ curve $(0, \epsilon_0) \ni \epsilon \mapsto p_\epsilon := (x_\epsilon, y_\epsilon) \in \mathbb{R}^n_{\gg 0} \times \mathbb{R}^k_{\gg 0}$ satisfying $\lim_{\epsilon \to 0+}(x_\epsilon, y_\epsilon) = (p,\underline{y})$ for some $p \in \mathbb{R}^n_{\gg 0}$, $\underline{y} \in \mathbb{R}^k_{\geq 0}$. Define $A_\epsilon := \epsilon D_xf(x_\epsilon, y_\epsilon; \epsilon)$ and $B_\epsilon := \epsilon D_yf(x_\epsilon, y_\epsilon; \epsilon)$. Assume that $\lim_{\epsilon \to 0+}A_\epsilon = 0$, and that $N_0:=\lim_{\epsilon\to 0+}B_\epsilon\beta$ is defined.
\begin{enumerate}[align=left,leftmargin=*]
\item[(a)] Suppose that (i) $J_{\mathrm{im}\,\Gamma}\Gamma v(p) \neq 0$, and (ii) $N_0$ is nonsingular. Then there exists $\epsilon'>0$ s.t., for $\epsilon \in (0,\epsilon')$, $J_{\mathrm{im}\,\Gamma'}F(x_\epsilon, y_\epsilon;\epsilon) \neq 0$.
\item[(b)] Suppose that (i) $D_{\mathrm{im}\,\Gamma}\Gamma v(p)$ is Hurwitz stable, and (ii) $N_0$ is Hurwitz stable. Then there exists $\epsilon'>0$ s.t., for $\epsilon \in (0,\epsilon')$, $D_{\mathrm{im}\,\Gamma'}F(x_\epsilon, y_\epsilon;\epsilon)$ is Hurwitz stable.
\end{enumerate}
\end{lemma1}

\begin{proof}
Define \[
\Gamma_0' := \left(\begin{array}{cc}\Gamma_0&\alpha\\0&\beta\end{array}\right), \quad Q' := \left(\begin{array}{cc}Q&0\\0&I\end{array}\right)
\]
and note that $\Gamma_0'$ has rank $r+m$ and that $\Gamma' = \Gamma_0'Q'$. Let $M_\epsilon := QDv(x_\epsilon)\Gamma_0$. We are interested in the reduced Jacobian matrix:
\[
D_{\Gamma_0'}F(x_\epsilon,y_\epsilon;\epsilon) = Q'\left(\begin{array}{ccc}Dv(x_\epsilon)&0\\\frac{1}{\epsilon}A_\epsilon&\frac{1}{\epsilon}B_\epsilon\end{array}\right)\Gamma_0' = \left(\begin{array}{cc}M_\epsilon&QDv(x_\epsilon)\alpha\\\frac{1}{\epsilon}A_\epsilon\Gamma_0&\frac{1}{\epsilon}A_\epsilon\alpha + \frac{1}{\epsilon}B_\epsilon\beta\end{array}\right)\,,
\]
which is a representative of $D_{\mathrm{im}\,\Gamma'}F(x_\epsilon, y_\epsilon;\epsilon)$.
\begin{enumerate}[align=left,leftmargin=*]
\item[(a)] If $M_0 := QDv(p)\Gamma_0$ is nonsingular then, as $v$ is $C^1$, there exists $\epsilon_1 > 0$ s.t. $M_\epsilon$ is nonsingular for $\epsilon \in (0,\epsilon_1)$. Then as $\epsilon \to 0+$, by the Schur determinant formula (p46 of \cite{gantmacher}, for example):
\[
\epsilon^mJ_{\Gamma_0'}F(x_\epsilon,y_\epsilon; \epsilon) = \mathrm{det}\,\left[A_\epsilon\alpha + B_\epsilon \beta- A_\epsilon\Gamma_0(M_\epsilon)^{-1}QDv(x_\epsilon)\alpha\right]\,\mathrm{det}(M_\epsilon) \to \mathrm{det}\,N_0\mathrm{det}\,M_0 \neq 0,
\]
as $N_0$ is nonsingular. Consequently, there exists $\epsilon' \leq \min\{\epsilon_0, \epsilon_1\}$ s.t. $J_{\mathrm{im}\Gamma'}F(x_\epsilon,y_\epsilon; \epsilon) \neq 0$ for $0<\epsilon < \epsilon'$. 
\item[(b)] If $N_0$ and $QDv(p)\Gamma_0$ are Hurwitz stable, then the assumptions of Lemma~\ref{lemHurwitz} are satisfied, and $D_{\mathrm{im}\Gamma'}F(x_\epsilon,y_\epsilon; \epsilon)$ is Hurwitz stable.
\end{enumerate}
This completes proof of the lemma.
\end{proof}

\begin{myproof}{Theorem~\ref{thmblockadd}}
Let the $m$ added reversible reactions be
\[
a_i\cdot X +b_i\cdot Y \rightleftharpoons a_i'\cdot X + b_i'\cdot Y,\quad (i = 1, \ldots, m)\,.
\]
Define $a = (a_1| a_2|\cdots|a_m)$, with $a'$, $b$ and $b'$ defined similarly. Define $\alpha = a'-a$ and $\beta = b'-b$. $\alpha$ is now an $n \times m$ matrix, and $\beta$ is a $k \times m$ matrix with rank $m$ by assumption, implying $k \geq m$. W.l.o.g., namely by reordering the species of $Y$ if necessary, let the first $m$ rows of $\beta$ form a nonsingular square matrix so that $\beta = \left(\begin{array}{c}\hat{\beta}\\\doublehat{\beta}\end{array}\right)$, where $\hat{\beta}$ is now a nonsingular $m \times m$ matrix, and $\doublehat{\beta}$ is a $(k-m) \times m$ matrix. We have
\[
\Gamma_0' := \left(\begin{array}{cc}\Gamma_0&\alpha\\0&\beta\end{array}\right),\quad Q' := \left(\begin{array}{cc}Q&0\\0&I\end{array}\right), \quad \mbox{and}\quad \Gamma' = \Gamma_0'Q' = \left(\begin{array}{cc}\Gamma&\alpha\\0&\beta\end{array}\right)\,,
\]
where $\Gamma'$ is the stoichiometric matrix of $\mathcal{R}'$ and $\Gamma_0'$ is a basis for $S':=\mathrm{im}\,\Gamma'$. Choose mass action kinetics for the added reactions and give the $j$th reaction forward and backward rate constants $k_j$ and $k_{-j}$ respectively. Setting $k_+ := (k_1, k_2, \ldots, k_m)^\mathrm{t}$ and $k_- := (k_{-1}, k_{-2}, \ldots, k_{-m})^\mathrm{t}$, the dynamics of the new system takes the form
\[
\left(\begin{array}{c}\dot x\\\dot y\end{array}\right) = \left(\begin{array}{cc}\Gamma&\alpha\\0&\beta\end{array}\right)\left(\begin{array}{c}v(x)\\\hat{f}(x,y)\end{array}\right)\,,\quad \mbox{where}\quad \hat{f}(x,y) := k_+\circ x^{a^\mathrm{t}}\circ y^{b^\mathrm{t}} - k_{-}\circ x^{{a'}^\mathrm{t}}\circ y^{{b'}^\mathrm{t}}\,.
\]
As $\beta$ has rank $m$ and hence represents an injective linear transformation, positive equilibria of $\mathcal{R}'$ are precisely the solutions of $\Gamma v(x)=0, \hat{f}(x,y)=0$, namely of $\Gamma v(x) = 0, \,\,y^{\beta^\mathrm{t}} = K \circ x^{-\alpha^\mathrm{t}}$, where $K = k_+/k_-$. Write $y = (\hat{y}, \doublehat{y}) \in \mathbb{R}^m \times \mathbb{R}^{k-m}$ so that $y^{\beta^\mathrm{t}} = K \circ x^{-\alpha^\mathrm{t}}$ can be written $\hat{y}^{\hat{\beta}^\mathrm{t}}\circ \doublehat{y}^{\hat{\hat{\beta}}^\mathrm{t}} = K\circ x^{-\alpha^\mathrm{t}}$. Taking logs gives:
\[
\hat{\beta}^\mathrm{t}\ln \hat{y} = \ln K -\alpha^\mathrm{t} \ln x - \doublehat{\beta}^\mathrm{t} \ln \doublehat{y}\,.
\]
We multiply through by $(\hat{\beta}^\mathrm{t})^{-1}$ which is defined by assumption, and let $\gamma := -(\alpha\,\hat{\beta}^{-1})^\mathrm{t}$, $\delta :=  -(\doublehat{\beta}\,\hat{\beta}^{-1})^\mathrm{t}$. Fix $k_+(\epsilon):=\bm{\epsilon}^{-\hat{b}^\mathrm{t}}$ and $k_-(\epsilon):=\bm{\epsilon}^{-{\hat{b}'}^\mathrm{t}}$ where $\hat{b}$ and $\hat{b}'$ refer to the top $m \times m$ submatrices of $b$ and $b'$ respectively. Setting $K=K(\epsilon):=k_+(\epsilon)/k_-(\epsilon) = \bm{\epsilon}^{\hat{\beta}^\mathrm{t}}$ and exponentiating again gives, at equilibrium:
\[
\hat{y} = \epsilon x^\gamma\circ \doublehat{y}^\delta\,.
\]
Define $g \colon \mathbb{R}^n_{\gg 0} \times \mathbb{R}^m \times \mathbb{R}^{k-m}_{\gg 0} \times \mathbb{R} \to \mathbb{R}^m$ by $g(x, \hat{y}, \doublehat{y};\epsilon) = \hat{y} - \epsilon x^\gamma\circ \doublehat{y}^\delta$. With $\epsilon>0$ fixed and rate constants chosen as above, positive equilibria of $\mathcal{R}'$ occur precisely at 
\[
0 = G(x, \hat{y}, \doublehat{y};\epsilon) := \left(\begin{array}{cc}\Gamma&\alpha\\0&\hat{\beta}\\0&\doublehat{\beta}\end{array}\right)\left(\begin{array}{c}v(x)\\ g(x, \hat{y}, \doublehat{y};\epsilon)\end{array}\right)\,.
\]
Note that $(p,0,\mathbf{1};0)$ and $(q,0,\mathbf{1};0)$ are in the domain of $G$ (taken to be that of $g$), which is open, and that $G(p,0,\mathbf{1};0) = G(q,0, \mathbf{1};0)=0$. We compute:
\begin{eqnarray*}
D_{\Gamma_0'}G(x,\hat{y}, \doublehat{y};\epsilon)
&=& \left(\begin{array}{cc}QDv(x)\Gamma_0&QDv(x)\alpha\\D_xg\Gamma_0&D_xg\alpha + D_{\hat{y}}g\hat{\beta}+D_{\hat{\hat{y}}}g\doublehat{\beta}\end{array}\right)\,.
\end{eqnarray*}
Now $D_{\hat{y}}g$ is the $m \times m$ identity matrix and we confirm that $D_xg(p, 0, \mathbf{1};0) = 0$ and $D_{\hat{\hat{y}}}g(p, 0, \mathbf{1};0) = 0$. So, by the nonsingularity of $QDv(p)\Gamma_0$ and of $\hat{\beta}$, 
\[
J_{S'}G(p,0, \mathbf{1};0) = \mathrm{det}(D_{\Gamma_0'}G(p,0, \mathbf{1};0)) = \mathrm{det}(QDv(p)\Gamma_0)\mathrm{det}\,\hat{\beta} \neq 0\,.
\]
Corollary~\ref{corift} gives, for sufficiently small $\epsilon>0$, a $C^1$ curve $\epsilon \mapsto p_\epsilon := (x(\epsilon), \hat{y}(\epsilon), \doublehat{y}(\epsilon))$ of zeros of $G$ on $(p,0,\mathbf{1})+S'$ such that $\lim_{\epsilon \to 0+}(x(\epsilon), \hat{y}(\epsilon), \doublehat{y}(\epsilon)) = (p,0,\mathbf{1})$. For sufficiently small $\epsilon > 0$, $x(\epsilon)\gg 0$ (as $x(0) = p \gg 0$) and $\doublehat{y}(\epsilon)\gg 0$ (as $\doublehat{y}(0) = \mathbf{1} \gg 0$), and consequently $\hat{y}(\epsilon) \gg 0$, since 
\begin{equation}
\label{eqyhat}
\hat{y}(\epsilon) = \epsilon x(\epsilon)^\gamma\circ \doublehat{y}(\epsilon)^\delta.
\end{equation}
Thus each $p_\epsilon$ is a positive equilibrium of $\mathcal{R}'$. An identical argument replacing $p$ with $q$ gives a curve of positive equilibria $q_\epsilon$ on $(q,0,\mathbf{1})+S'$ approaching $(q,0,\mathbf{1})$ as $\epsilon \to 0+$. Since $(p,0,\mathbf{1}) - (q, 0, \mathbf{1}) = (p-q, 0, 0) \in S'$, $(p,0,\mathbf{1})$ and $(q, 0, \mathbf{1})$ are compatible, and hence $p_\epsilon$ and $q_\epsilon$ are compatible. 

We now infer nondegeneracy of $p_\epsilon$ using Lemma~\ref{lemnondegen}(a). Define 
\[
f(x,y;\epsilon):=k_+(\epsilon)\circ x^{a^\mathrm{t}}\circ y^{b^\mathrm{t}} - k_{-}(\epsilon)\circ x^{{a'}^\mathrm{t}}\circ y^{{b'}^\mathrm{t}} \quad \mbox{and} \quad F(x,y;\epsilon) := \Gamma'\left(\begin{array}{c}v(x)\\f(x,y;\epsilon)\end{array}\right)\,.
\] 
Define $T\colon (0, \epsilon_0) \to \mathbb{R}^m$ by $T(\epsilon) = k_+(\epsilon)\circ x(\epsilon)^{a^\mathrm{t}}\circ y(\epsilon)^{b^\mathrm{t}}$. Then, using (\ref{eqyhat}), we compute that $T_0:=\lim_{\epsilon\to 0+}T(\epsilon) = p^{a^\mathrm{t}+\hat{b}^\mathrm{t}\gamma} \gg 0$. Observe that 
\[
\begin{array}{l}
A_\epsilon:=\epsilon D_xf(x(\epsilon), y(\epsilon); \epsilon) = -\epsilon \mathrm{diag}(T(\epsilon))\alpha^\mathrm{t}\mathrm{diag}(\mathbf{1}/x(\epsilon)),\\
B_\epsilon:=\epsilon D_yf(x(\epsilon), y(\epsilon); \epsilon) = -\mathrm{diag}(T(\epsilon))\beta^\mathrm{t}\mathrm{diag}(\bm{\epsilon}/y(\epsilon)).
\end{array}
\] 
Set $N_1:=\mathrm{diag}(T_0)$, $C_0 := \mathrm{det}(N_1)$, $N_2:= \lim_{\epsilon \to 0+}\mathrm{diag}(\bm{\epsilon}/y(\epsilon))$ (which, by (\ref{eqyhat}), is defined and nonnegative), and $N_0:=\lim_{\epsilon \to 0+}B_\epsilon\beta =-N_1\beta^\mathrm{t}N_2\beta$. Note that $C_0 > 0$. We compute that, as required in Lemma~\ref{lemnondegen}, $A_\epsilon \to 0 \,\,\mbox{as}\,\, \epsilon \to 0+$ since $\mathrm{diag}(T(\epsilon)) \to N_1$ as $\epsilon \to 0+$ and $\mathrm{diag}(\mathbf{1}/x(\epsilon)) \to \mathrm{diag}(\mathbf{1}/p)$ as $\epsilon\to 0+$. By (\ref{eqyhat}), $K' :=N_2[\langle m \rangle] = p^{-\mathbf{1}^\mathrm{t}\gamma} > 0$, while $N_2[\theta] = 0$ for $\theta \subseteq \langle k \rangle$, $|\theta|=m$, $\theta \neq \langle m \rangle$. Applying the Cauchy-Binet formula (\cite{gantmacher} for example) to $N_0$ gives,
\[
\mathrm{det}(N_0) = (-1)^mC_0\sum_{\theta \subseteq \langle k \rangle,\,\, |\theta| = m}N_2[\theta]\beta[\theta|\langle m \rangle]^2 = (-1)^mC_0K'\beta[\langle m \rangle]^2\neq 0,
\]
and so $N_0$ is nonsingular. All the conditions of Lemma~\ref{lemnondegen}(a) are met, and we conclude that for each sufficiently small $\epsilon>0$, $p_\epsilon$ is a nondegenerate equilibrium of $\mathcal{R}'$. An identical argument applies to $q_\epsilon$. Thus $\mathcal{R}'$ admits MPNE.

To see inheritance of MPSE observe that $N_0 = -N_1^{1/2}AA^\mathrm{t}N_1^{-1/2}$, where $A := N_1^{1/2}\beta^\mathrm{t}N_2^{1/2}$. $N_0$ is hence similar to the negative semidefinite matrix $-AA^\mathrm{t}$. As $N_0$ has already been shown to be nonsingular, it must in fact be Hurwitz. If $p$ is linearly stable for $\mathcal{R}$, namely $D_{\mathrm{im}\,\Gamma}\Gamma v(p)$ is Hurwitz, then the conditions of Lemma~\ref{lemnondegen}(b) are met, and we conclude that for each sufficiently small $\epsilon$, $p_\epsilon$ is linearly stable for $\mathcal{R}'$. An identical argument applies to $q_\epsilon$ and so, if $\mathcal{R}$ admits MPSE, then so does $\mathcal{R}'$.
\end{myproof}

\begin{myproof}{Theorem~\ref{thmintermediates}}
The proof is similar to that of Theorem~\ref{thmblockadd}. Recall that we create $\mathcal{R}'$ from $\mathcal{R}$ by replacing each of the  $m$ reactions:
\[
a_i \cdot X \rightarrow b_i \cdot X\quad \mbox{with a chain}\quad a_i \cdot X \rightarrow c_i \cdot X + \beta_i\cdot Y \rightarrow b_i \cdot X,\,\,(i=1,\ldots,m)\,.\]
As we are assuming that $\beta:=(\beta_1|\beta_2|\cdots|\beta_m)$ has rank $m$, as in the proof of Theorem~\ref{thmblockadd} we may write $\beta = \left(\begin{array}{c}\hat{\beta}\\\doublehat{\beta}\end{array}\right)$ where $\hat{\beta}$ is a nonsingular $m \times m$ matrix, and $\doublehat{\beta}$ is a $(k-m) \times m$ matrix. Also w.l.o.g (namely, by writing a reaction as multiple reactions and/or reordering the reactions of $\mathcal{R}$ as necessary) let the reaction $a_i \cdot X \rightarrow b_i\cdot X$ figure as the $i$th reaction in the original CRN, proceeding with rate $v_i(x)$, and let $\underline{v}(x) = (v_1(x), \ldots, v_m(x))^\mathrm{t}$. Set the rate of each new reaction $a_i \cdot X \rightarrow c_i \cdot X + \beta_i \cdot Y$ to be $v_i(x)$ (this is consistent with mass action kinetics if the original reaction had mass action kinetics) and set the second added reaction in the $i$th chain to have mass action kinetics with rate constant $k_i$ ($i = 1, \ldots, m$). Define $k := (k_1, \ldots, k_m)^\mathrm{t}$ and $c:=(c_1|c_2|\cdots|c_m)$. Note that for $x \gg 0$, $v_i(x) \gg 0$ under the assumption of positive general kinetics, and in particular $v_i(p) \gg 0$. Define $\alpha_i := c_i - b_i$, $(i = 1, \ldots, m)$, $\alpha := (\alpha_1|\alpha_2|\cdots |\alpha_m)$,
\[
\hat{f}(x,y) := \underline{v}(x) - k \circ x^{c^\mathrm{t}} \circ y^{\beta^\mathrm{t}},\quad\Gamma_0' := \left(\begin{array}{cc}\Gamma_0&\alpha\\0&\beta\end{array}\right)\quad \mbox{and}\quad Q' := \left(\begin{array}{cc}Q&0\\0&I\end{array}\right)
\]
so that $\Gamma' = \Gamma_0'Q'$. $\Gamma'$ is the stoichiometric matrix of $\mathcal{R}'$ and the columns of $\Gamma_0'$ are a basis for $S':=\mathrm{im}\,\Gamma'$. The dynamics of $\mathcal{R}'$ is governed by
\[
\left(\begin{array}{c}\dot x\\\dot y\end{array}\right) = \Gamma'\left(\begin{array}{c}v(x)\\\hat{f}(x,y)\end{array}\right) = \left(\begin{array}{cc}\Gamma&\alpha\\0&\beta\end{array}\right)\left(\begin{array}{c}v(x)\\\hat{f}(x,y)\end{array}\right)\,.
\]
Thus, with the kinetic choices made so far, the net effect on the vector field of ``inserting'' the complexes $c_i \cdot X + \beta_i\cdot Y$ ($i=1,\ldots,m$) is the same adding to $\mathcal{R}$ $m$ new (pseudo)-reactions with stoichiometric matrix $\left(\begin{array}{c}\alpha\\\beta\end{array}\right)$ and rate vector $\underline{v}(x)-k \circ x^{c^\mathrm{t}}\circ y^{\beta^\mathrm{t}}$. As $\beta$ has rank $m$ (and hence corresponds to an injective linear transformation), positive equilibria of $\mathcal{R}'$ are solutions of $\Gamma v(x) = 0, \,\,\hat{f}(x,y)=0$. Write $y = (\hat{y}, \doublehat{y}) \in \mathbb{R}^m \times \mathbb{R}^{k-m}$. The condition $\hat{f}(x,y)=0$ reads $\underline{v}(x)-k \circ x^{c^\mathrm{t}}\circ \hat{y}^{\hat{\beta}^\mathrm{t}}\circ \doublehat{y}^{\hat{\hat{\beta}}^\mathrm{t}}=0$. We now perform some manipulations similar to those in the proof of Theorem~\ref{thmblockadd}. Setting $k = k(\epsilon):= \bm{\epsilon}^{-\hat{\beta}^\mathrm{t}}$, and defining $\gamma := -(c(\hat{\beta})^{-1})^\mathrm{t}$, $\delta := -(\doublehat{\beta}(\hat{\beta})^{-1})^\mathrm{t}$, and $V(x) := (\underline{v}(x))^{(\hat{\beta}^{-1})^\mathrm{t}}$, we end up with, at equilibrium, 
\[
\hat{y} = \epsilon V(x) \circ x^\gamma \circ \doublehat{y}^\delta\,.
\]
Define $g \colon \mathbb{R}^n_{\gg 0} \times \mathbb{R}^m \times \mathbb{R}^{k-m}_{\gg 0} \times \mathbb{R} \to \mathbb{R}^m$ by $g(x, \hat{y}, \doublehat{y}; \epsilon) = \hat{y} - \epsilon V(x) \circ x^\gamma \circ \doublehat{y}^\delta$. Then, for $\epsilon > 0$ and reaction rates chosen as above, positive equilibria of $\mathcal{R}'$ occur precisely at 
\[
0 = G(x, \hat{y}, \doublehat{y};\epsilon) := \left(\begin{array}{cc}\Gamma&\alpha\\0&\hat{\beta}\\0&\doublehat{\beta}\end{array}\right)\left(\begin{array}{c}v(x)\\ g(x, \hat{y}, \doublehat{y};\epsilon)\end{array}\right)\,.
\]
Note that $(p,0,\mathbf{1};0)$ and $(q,0,\mathbf{1};0)$ are in the domain of $G$ (taken to be that of $g$), which is open, and that $G(p,0,\mathbf{1};0) = G(q,0, \mathbf{1};0)=0$. We compute:
\[
D_{\Gamma_0'}G(x,\hat{y}, \doublehat{y};\epsilon) =  \left(\begin{array}{cc}QDv(x)\Gamma_0&QDv(x)\alpha\\D_xg\Gamma_0&D_xg\alpha + D_{\hat{y}}g\hat{\beta}+D_{\hat{\hat{y}}}g\doublehat{\beta}\end{array}\right)\,.
\]
Now $D_{\hat{y}}g = I$, $D_xg(p, 0, \mathbf{1};0) = 0$, $D_{\hat{\hat{y}}}g(p, 0, \mathbf{1};0) = 0$, and so, as $\hat{\beta}$ and $QDv(p)\Gamma_0$ are nonsingular, $J_{S'}G(p,0, \mathbf{1};0) = \mathrm{det}(QDv(p)\Gamma_0)\mathrm{det}(\hat{\beta}) \neq 0$. Corollary~\ref{corift} gives, for sufficiently small $\epsilon>0$, a curve $\epsilon \mapsto p_\epsilon := (x(\epsilon), \hat{y}(\epsilon), \doublehat{y}(\epsilon))$ of zeros of $G$ on $(p,0,\mathbf{1})+S'$ such that $\lim_{\epsilon \to 0+}(x(\epsilon), \hat{y}(\epsilon), \doublehat{y}(\epsilon)) = (p,0,\mathbf{1})$. For sufficiently small $\epsilon>0$, $x(\epsilon)\gg 0$, $V(x(\epsilon)) \gg 0$ (as $x(\epsilon) \to p \gg 0$) and $\doublehat{y}(\epsilon)\gg 0$ (as $\doublehat{y}(\epsilon) \to \mathbf{1} \gg 0$), and consequently $\hat{y}(\epsilon) \gg 0$ as
\begin{equation}
\label{eqyhat1}
\hat{y}(\epsilon) = \epsilon V(x(\epsilon)) \circ x(\epsilon)^\gamma \circ \doublehat{y}(\epsilon)^\delta.
\end{equation}
Thus the curve $p_\epsilon$ consists of positive equilibria of $\mathcal{R}'$. An identical argument replacing $p$ with $q$ gives positive equilibria $q_\epsilon$ on $(q,0,\mathbf{1})+S'$ approaching $(q,0,\mathbf{1})$ as $\epsilon \to 0+$. Since $(p,0,\mathbf{1}) - (q, 0, \mathbf{1}) = (p-q, 0, 0) \in S'$, $p_\epsilon$ and $q_\epsilon$ are compatible. 

We can now infer nondegeneracy of $p_\epsilon$ and $q_\epsilon$ using Lemma~\ref{lemnondegen}(a). Set
\[
f(x,y;\epsilon) := \underline{v}(x) - k(\epsilon) \circ x^{c^\mathrm{t}} \circ y^{\beta^\mathrm{t}} \quad \mbox{and}\quad F(x,y;\epsilon):= \Gamma'\left(\begin{array}{c}v(x)\\f(x,y;\epsilon)\end{array}\right).
\]
Define
\[
\begin{array}{l}
A_\epsilon := \epsilon D_xf(x(\epsilon), y(\epsilon);\epsilon)=\epsilon\left[D\underline{v}(x(\epsilon)) - \mathrm{diag}(\underline{v}(x(\epsilon)))c^\mathrm{t}\mathrm{diag}(\mathbf{1}/x(\epsilon))\right]\,,\\
B_\epsilon:=\epsilon D_yf(x(\epsilon), y(\epsilon); \epsilon) = -\mathrm{diag}(\underline{v}(x(\epsilon)))\beta^\mathrm{t}\mathrm{diag}(\bm{\epsilon}/y(\epsilon))\,.
\end{array}
\]
Set $N_1:=\mathrm{diag}(\underline{v}(p))$, $C_0 := \mathrm{det}(N_1)$ and $N_2:= \lim_{\epsilon \to 0+}\mathrm{diag}(\bm{\epsilon}/y(\epsilon))$ (which, by (\ref{eqyhat1}), is defined and nonnegative). Note that $C_0 > 0$. It is clear that $A_\epsilon \to 0$ as $\epsilon \to 0+$ as the quantity in the square bracket approaches the constant matrix $\left[D\underline{v}(p) - N_1c^\mathrm{t}\mathrm{diag}(\mathbf{1}/p)\right]$ as $\epsilon \to 0+$. From (\ref{eqyhat1}), $K':=N_2[\langle m \rangle]$ is a positive constant, while $N_2[\theta] = 0$ for $\theta \subseteq \langle k \rangle$, $|\theta|=m$, $\theta \neq \langle m \rangle$. Then $N_0:=\lim_{\epsilon \to 0+}B_\epsilon\beta = -N_1\beta^\mathrm{t}N_2\beta$, and by the Cauchy-Binet formula,
\[
\mathrm{det}(N_0) =  (-1)^mC_0\sum_{\theta \subseteq \langle k \rangle,\,\, |\theta| = m}N_2[\theta](\beta[\theta|\langle m \rangle])^2 \to (-1)^mC_0K'\beta[\langle m \rangle]^2\neq 0,
\]
and so $N_0$ is nonsingular. All the conditions of Lemma~\ref{lemnondegen}(a) are met, and we conclude that for each sufficiently small $\epsilon>0$, $p_\epsilon$ is a nondegenerate equilibrium of $\mathcal{R}'$. An identical argument applies to $q_\epsilon$.

Almost identically to the calculation in the proof of Theorem~\ref{thmblockadd}, $N_0$ can be seen to be Hurwitz. If $p$ is linearly stable for $\mathcal{R}$, namely $D_{\mathrm{im}\,\Gamma}\Gamma v(p)$ is Hurwitz, then by Lemma~\ref{lemnondegen}(b), for each sufficiently small $\epsilon>0$, $p_\epsilon$ is linearly stable with respect to $\mathrm{im}\,\Gamma'$. An identical argument applies to $q_\epsilon$, and so if $\mathcal{R}$ admits MPSE, then so does $\mathcal{R}'$.
\end{myproof}

\begin{remark}[Kinetic assumptions in Theorems~\ref{thmblockadd}~and~\ref{thmintermediates}]
\label{remkin}
Examining the proofs of Theorems~\ref{thmblockadd}~and~\ref{thmintermediates} we see that mass action kinetics does not really play a fundamental role. Most important are the equations (\ref{eqyhat}) and (\ref{eqyhat1}), namely that we can sufficiently control the rates of the added reactions to ensure that the $\hat{y}$ values at the new equilibria approach $0$ as some parameter $\epsilon \to 0+$. The results admit generalisation in this direction.
\end{remark}

\section{Conclusions}

We have begun to describe how a CRN may be enlarged while maintaining the properties of having multiple positive nondegenerate equilibria (MPNE), or multiple positive linearly stable equilibria (MPSE). The modifications in Theorems~\ref{thmnewdepreac}~to~\ref{thmintermediates} collectively define a partial order $\preceq$ on the set of all CRNs. In the biologically motivated example of Section~\ref{secbioexample} we identify an ``atom of MPSE'', namely a CRN which admits MPSE and which is minimal w.r.t. $\preceq$, and show how this atom occurs in various published models, immediately implying MPSE in these models. We suspect that often the small multistationary networks described in \cite{JoshiShiu2016}, minimal w.r.t. the induced subnetwork partial order, are also minimal w.r.t. $\preceq$, and hence form natural building-blocks of CRNs displaying MPNE. This remains to be confirmed.

Our approach was local: the main theoretical tool used was the IFT and, informally speaking, we mostly constructed $\mathcal{R}'$ as a {\em perturbation} of $\mathcal{R}$ in some sense. The results presented here certainly do not exhaust the possibilities in this approach. We have chosen not to include a number of partial or isolated results in the same vein as Theorems~\ref{thmblockadd}~and~\ref{thmintermediates}. These theorems themselves can probably be strengthened; for example, with some added effort, reversibility of the added reactions in Theorem~\ref{thmblockadd} can probably be replaced with a weak reversibility requirement. It also seems that insisting that the original CRN should have mass action kinetics may allow modifications which our less restrictive assumptions would not permit because it forces certain relationships to hold between reaction rates: in particular, the ratio of the rates of two reactions with the same source complex would be constant. 

Going beyond the techniques here, it seems likely that algebraic approaches may reveal network modifications which preserve MPE or MPNE, but which could not be predicted using a local approach: some of the modifications in Section~\ref{secextended} and the results in \cite{feliuwiufInterface2013} point in this direction. Ultimately the existence of MPE for mass action systems is about the cardinality of the intersection of certain algebraic varieties: it may be that adding new species or reactions in certain ways can be proved to allow MPE, but not by a local approach. Working towards the ``best'' partial order on CRNs from the point of view of the inheritance of MPNE is a project to be continued.

\appendix

\section{MPNE in the minimal MAPK model of Section~\ref{secbioexample}}
\label{appbio}

The evolution of $x_1:=[\text{MKK-pp}]$, $x_2:=[\text{MAPK}]$, $x_3:=[\text{MKK-pp--MAPK}]$, $x_4:=[\text{MAPK-p}]$ and $x_5:=[\text{MAPK-pp}]$ is governed by the ODE system:
\[
\begin{array}{lll}
\dot x_1 = -k_1x_1x_2+(k_2+k_3)x_3, & \dot x_2 =-k_1x_1x_2+k_2x_3+k_6x_4, & \\\dot x_3 = k_1x_1x_2-(k_2+k_3)x_3,&
\dot x_4 =k_3x_3-k_4x_1x_4+k_5x_5-k_6x_4, & \dot x_5 =k_4x_1x_4-k_5x_5\,. 
\end{array}
\]
which has precisely two independent conserved quantities
$M=x_1+x_3\ge 0$ and $N=x_2+x_3+x_4+x_5\ge 0.$
At steady state, one can easily express all variables in terms of $x_2$. We omit the calculations, but note that positive values of $x_2$ will automatically give rise to positive values for all variables except $x_5$. Substituting these in $\dot x_5=0$ yields a cubic equation in $x_2$:
\begin{eqnarray}\label{eq:univariate}
f(x_2) &:=& -k_1^2 k_5 k_6 x_2^3 + [k_1^2 k_5k_6(N -M) -  k_1^2 k_3 k_5 M - 2k_1k_5k_6(k_2+k_3)]x_2^2  \\ \nonumber                                    
&&+ [2N{ k_1}{ k_3}{ k_5}{ k_6}-{ k_3}^2{ k_5}{ k_6}-2{ k_2}{ k_3}
 { k_5}{ k_6}-M{ k_1}{ k_3}{ k_5}{ k_6}-{ k_2}^2
 { k_5}{ k_6}\\ \nonumber
 &&+2N{ k_1}{ k_2}{ k_5}
 { k_6}
 -M{ k_1}{ k_2}{ k_5}{ k_6}-M{ k_1}
 { k_3}^2{ k_5}-M{ k_1}{ k_2}{ k_3}
 { k_5}-M^2{ k_1}{ k_3}^2{ k_4}\\ \nonumber
 &&-M^2{ k_1}
 { k_2}{ k_3}{ k_4} ]x_2+
 k_5k_6(k_2+k_3)^2N = 0
\end{eqnarray}

Now we impose alternate signs on the coefficients of $f$ (a necessary condition for existence of three positive roots), and proceed by trial and error to produce parameter values that give rise to distinct positive solutions, and which make $x_5>0.$ For example, $k_1=k_2=k_3=k_4=k_6=1, k_5=3, M=100, N=600$ results in three positive steady states computed numerically as
\[
(x_1, x_2, x_3, x_4, x_5)=\left\{\begin{array}{l}(0.5183, 383.8, 99.48, 99.48, 17.19),\\(14.69, 11.61, 85.31, 85.31, 417.8),\\(78.79, 0.5384, 21.21,  21.21, 557.0).\end{array}\right.
\]
and a short calculation of the reduced Jacobian shows that all of them are nondegenerate. It was checked numerically that the first and the third steady state are linearly stable w.r.t. their stoichiometry class (whereas the second is not) and thus this CRN admits MPSE.

\bibliographystyle{unsrt}

\end{document}